\documentclass[mathpazo]{csiam-am}
\renewcommand\thisnumber{x}
\renewcommand\thisyear {20xx}
\renewcommand\thismonth{xx}
\renewcommand\thisvolume{x}

\RequirePackage{xcolor}
\usepackage{appendix}
\usepackage{subfigure}
\usepackage{float}
\usepackage{graphicx}
\usepackage{epstopdf}
\usepackage{cases}

\def\x{\vec{x}}

\begin{document}
\title{A Variational Discretization Method for Mean\\ 
	Curvature Flows by the Onsager Principle}

\author[O.~Author]{Only Author\corrauth}
\address{School of Mathematical Sciences, Beijing Normal University,
Beijing 100875, P.R. China}

\author[F.~Author and A.~Co-Author(s)]{First Author\affil{1}\comma\corrauth\ and Co-Author(s)\affil{2}}
\address{\affilnum{1}\ address of First Author\\
\affilnum{2}\ address of Co-Author(s)}

\author[Y. Liu and X. Xu,]{Yihe Liu and Xianmin Xu\corrauth}
\address{LSEC, ICMSEC, NCMIS, Academy of Mathematics and Systems Science, Chinese Academy of Sciences, Beijing 100190, China}

\emails{{\tt liuyihe@amss.ac.cn} (Y. Liu), {\tt xmxu@lsec.cc.ac.cn} (X. Xu)}

\begin{abstract}
The mean curvature flow describes the evolution of a surface (a curve) with normal velocity proportional to the local mean curvature. It has many applications in mathematics, science and engineering. In this paper, we develop a numerical method for mean curvature flows by using the Onsager principle as an approximation tool. We first show that the mean curvature flow can be derived naturally from the Onsager variational principle. Then we consider a piecewise linear approximation of the curve and derive a discrete geometric flow. The discrete flow is described by a system of ordinary differential equations for the nodes of the discrete curve. We prove that the discrete system preserve the energy dissipation structure in the framework of the Onsager principle and this implies the energy decreasing property. The ODE system can be solved by the improved Euler scheme and this leads to an efficient fully discrete scheme. We first consider the method for a simple mean curvature flow and then extend it to the volume preserving mean curvature flow and also a wetting problem on substrates. Numerical examples show that the method has optimal convergence rate and works well for all the three problems.
\end{abstract}

\maketitle


\section{Introduction}
Mean curvature flow describes the process of surface evolution, where a surface moves in its normal direction with a velocity equal to its mean curvature, i.e.
\begin{equation}
	\vec{v}=-H\vec{n},
\end{equation}
where $\vec{v}$ denotes the velocity of motion of the surface, and $H$ denotes the mean curvature of the surface. $\vec{n}$ is the unit normal vector of the surface. The mean curvature flow first appeared in materials science in the 1920s, and the general mathematical model was first proposed in 1957 in \cite{mullins1957theory}, where Mullins used it to describe how grooves develop on the surfaces of hot polycrystals. The mean curvature flow has found numerous applications. One is in image smoothing, where it is used to improve image quality by removing unnecessary noise and enhancing specific features without altering the transmitted information \cite{gallagher2020curve}. Additionally, the mean curvature flow has been employed to model a special form of the reaction-diffusion equation in chemical reactions \cite{rubinstein1989fast}. 

The numerical approximation of mean curvature flow originates from the pioneering work of Dziuk in 1990 \cite{dziuk1990algorithm}. He proposed a parametric finite element method (PFEM) to solve the problem based on the observation that the mean curvature flow is a diffusion equation for the surface embedding in a bulk region. However, as the surface evolves, the nodes in the PFEMs may overlap, leading to mesh distortion \cite{hu2022evolving}. Subsequently, various techniques have been introduced in PFEMs that allow points on interpolated curves or surfaces to shift tangentially to produce a better distribution of nodes. For example, Deckelnick and Dziuk introduced an artificial tangential velocity in their work \cite{deckelnick1995approximation}. By similar motivation, Elliott and Fritz employed a DeTurck's trick to introduce a tangential velocity  \cite{m2017approximations}. A widely used method is the parametric finite element method developed by Barrett, Garcke, and $\text{N\"{u}rnberg}$ \cite{Barrett2007OnTV, Barrett2008OnTP, Barrett2008ParametricAO}. They utilize a different variational form which can induce tangential velocity for the mesh nodes automatically.  The method have been generalized some other geometrical flows \cite{Barrett2010NumericalAO, barrett2011approximation,barrett2012parametric,jiang2021perimeter,bao2021structure,li2021energy,bao2022volume,bao2023symmetrized,hu2022evolving}. Recently, there are  significant progresses in rigorous convergence analysis for discrete schemes for mean curvature flows and some other geometric flows \cite{kovacs2019convergent,kovacs2021convergent,li2020convergence,li2021convergence,Bai2023ANA}. In addition,  there exist many other methods to approximate the mean curvature flow numerically, such as the threshold dynamics method \cite{ merriman1994motion, ruuth2003simple, esedoḡ2015threshold, Xu2016AnET, Wang2019AnIT}, the level set method \cite{osher1988fronts, evans1991motion, chen1991uniqueness, sethian1999level, zhao1996variational, saye2020review}, and the phase field method \cite{Edwards1991InterfacialTP, evans1992phase, eggleston2001phase, feng2003numerical}, etc.

In this paper, we develop a new  finite element method for mean curvature flow by using the Onsager principle as an approximation tool. We consider three different problems, namely the standard mean curvature flow, the volume preserving mean curvature flow and a wetting problem with contact with substrate. We show that the dynamic equations of all the three problems can be derived naturally by the Onsager principle. Furthermore, by considering a piecewise linear approximation of the surface, we can derive semi-discrete problems which preserve the energy dissipation relations just as the continuous problems. The discrete volume can also be preserved if the continuous problem preserve the volume. The semi-discrete problems can be further discretized in time by the improved Euler scheme. Numerical examples show that the method works well for all the three problems. In particular, it has optimal convergence rate in space for a mean curvature flow with analytic solutions. We mainly consider  mean curvature flows of a curve in two dimensional space, also known as curve shortening flow in literature. Our method can be generalized to higher dimensional problems  while there will be more restrictions on the time discrete schemes to avoid mesh distortion. 

The structure of the paper is as follows. In Section 2, we introduce the main idea to use the Onsager variational principle as an approximation tool. In Section 3, we first derive a continuous partial differential equation for a simple smooth closed curve under mean curvature flow using the Onsager principle. Subsequently, we discretize the equation using piecewise linear curves and prove the discrete energy decays property. We introduce a penalty term to the energy functional to avoid mesh distortion. We also briefly discuss the extension of the method to higher dimensional cases. In Section 4, we further consider a mean curvature flow with the constraint of volume conservation. We derive a continuity equation and also a discrete problem. In Section 5, we further generalize the method to solve a wetting problem, which is a non-closed curve contacting with substrates. In Section 6, we present several numerical examples to demonstrate the efficiency of the methods. We give a few conclusion remarks in the final section.

\section{The Onsager principle as an approximation tool}

Consider a physical system which is described by a time-dependent function $u(t)$. For simplication in notation, we denote the time derivative of $u$ by $\dot{u}=\frac{\partial u}{\partial t}$. In a dissipative system, the evolution of $u$ over time may cause energy dissipation, which can be quantified by a dissipation function denoted by $\Phi(\dot{u})$. In many applications, the dissipation function can be expressed as    
\begin{equation*}
	\Phi(\dot{u})=\frac{\xi}{2}\left\Vert{\dot{u}}\right\Vert_2^2, 
\end{equation*}
where $\xi$ is a positive friction coefficient. Suppose the free energy of the system is given by a function $\mathcal{E}(u)$. Given $\dot{u}$, the changing rate of the total energy can be computed as
\begin{equation*}
	\begin{aligned}
		\dot{\mathcal{E}}(u;\dot{u})
		=\left \langle\frac{\delta \mathcal{E}(u)}{\delta u},\dot{u} \right \rangle
		=\left . \frac{\text{d} \mathcal{E}(u+\lambda\dot{u})}{\text{d}\lambda}\right \vert_{\lambda=0}.
	\end{aligned}
\end{equation*}
Then we can define a Rayleighian functional as follows
\begin{equation} \label{Rayleighian}
	\mathcal{R}(u;\dot{u})=\Phi(\dot{u})+\dot{\mathcal{E}}(u;\dot{u}).
\end{equation}

Based on the above aforementioned definitions, the Onsager variational principle can be stated as follows \cite{Onsager1931,Onsager1931a,DoiSoft}. For any dissipative system without inertial effect, the evolution of the system can be obtained by minimizing the Rayleighian with respect to $\dot{u}$. In other words, the evolution equation of $u$ is obtained by
\begin{equation}
	\min_{\dot{u}\in V}\mathcal{R}(u,\dot{u}),\label{minR}
\end{equation}
where $V$ represents the admissible space of $\dot{u}$. Since the Rayleighian is a quadratic form with respect to $\dot{u}$, the minimization problem \eqref{minR} can be described by the Euler-Lagrange equation
\begin{equation}
	( \xi\dot{u},\varphi)=-\left \langle\frac{\delta\mathcal{E}(u)}{\delta u},\varphi \right \rangle,\quad \forall\varphi \in V, \label{E-L equation}
\end{equation}
or in a simple form
\begin{equation}
	\xi\dot{u}=-\frac{\delta\mathcal{E}(u)}{\delta u}.\label{E-L simple form}
\end{equation}
In physics, Eq. \eqref{E-L simple form} implies a balance between the frictional force $-\xi\dot{u}$ and the driven force $\frac{\delta\mathcal{E}(u)}{\delta u}$. For a more comprehensive introduction of the Onsager principle, we refer to \cite{DoiSoft}.

Recently, the Onsager principle is used to derive an approximate model for a physical system \cite{Doi2015}. The key idea is follows. Suppose that the system can be described by 
a system of slow variables $\mathbf{a}=(a_1,a_2,\cdots, a_N)$. Suppose we can compute the approximate energy $\mathcal{E}_h(\mathbf{a})$ and the approximate dissipation function $\Phi_h(\dot{\mathbf{a}})$. Then we can derive a discrete ODE system by the Onsager principle, which is given by 
$$
\frac{\delta \Phi_h}{\delta \dot{\mathbf a}}+\frac{\delta \mathcal{E}_h}{\delta \mathbf{a}}=0.
$$
The method has been used to develop reduced models in complicated two-phase flows \cite{XuXianmin2016,DiYana2016}, some soft matter 
problems \cite{ManXingkun2016,ZhouJiajia2017},  solid dewetting \cite{Jiang19b},  and also in dislocation dynamics \cite{dai2021boundary}. It can also be used to derive numerical method for wetting problems \cite{lu2021efficient}. In this following, we will investigate the application of the Onsager principle in deriving numerical schemes for mean curvature flows.

\section{Mean curvature flow} \label{problem 1}
\subsection{Derivation of the continuous equation by the Onsager principle} \label{general CSF model}
We mainly consider the mean curvature flow in two dimensional space. For a simple smooth closed curve $\Gamma(t)$ in the plane, the points on it are denoted by $\x(t)$. We aim to derive the mean curvature flow equation of $\Gamma(t)$ by using the Onsager variational principle.

First, suppose the dimensionless surface energy in a physical system is given by 
\begin{equation}
	\mathcal{E}=\int_{\Gamma(t)}\text{d}s=|\Gamma(t)|,
\end{equation}
where $s$ is the arc length parameter. In order to calculate the changing rate of the surface energy with respect to time, we need to employ the following transport equation for a closed curve \cite{barrett2020parametric}
\begin{equation}\label{curve RTE}
		\frac{\text{d}}{\text{d}t}\int_{\Gamma(t)} f \text{d}s = \int_{\Gamma(t)} \left (\frac{\partial f}{\partial t} + v_n\vec{n}\cdot\nabla f + f\kappa v_n\right ) \text{d}s.
\end{equation}
Where $\vec{v}=\frac{\text{d}\x}{\text{d}t}$, $v_n=\vec{v}\cdot\vec{n}$ is the normal velocity, and $\kappa$ denotes the curvature of $\Gamma(t)$. Let $f=1$ in \eqref{curve RTE}, we obtain the changing rate of the total energy with respect to time
\begin{equation} \label{dE/dt}
	\begin{aligned}
		\dot{\mathcal{E}}=\frac{\text{d}}{\text{d}t}\int_{\Gamma(t)}\text{d}s=\int_{\Gamma(t)}\kappa v_{n}\text{d}s. 
	\end{aligned}
\end{equation}
Suppose that the dissipation function in the system is given by
\begin{equation*}   
	\Phi=\frac{1}{2}\int_{\Gamma(t)} |\vec{v}|^2 \text{d}s.  
\end{equation*}
Then the corresponding Rayleighian is defined as
\begin{equation*}
	\mathcal{R}=\Phi+\dot{\mathcal{E}}=\int_{\Gamma(t)}\left (\frac{1}{2}|\vec{v}|^2+\kappa v_{n}\right )\text{d}s.
\end{equation*}
Using the Onsager principle, we minimize the Rayleighian $\mathcal{R}$ with respect to $\vec{v}$, i.e.
\begin{equation*}
	\min_{\vec{v}}\mathcal{R}. \label{minR v_n}
\end{equation*}
The corresponding Euler-Lagrange equation is
\begin{equation}
	\vec{v}=-\kappa\vec{n}. \label{mean curvature E-L eqaution}
\end{equation}
This is the mean curvature flow \cite{li2022parametric}. Substituting Eq. \eqref{mean curvature E-L eqaution} into \eqref{dE/dt}, we have 
\begin{equation*} 
	\dot{\mathcal{E}}
	=\int_{\Gamma(t)}\kappa v_{n}\text{d}s
	=-\int_{\Gamma(t)}\kappa^2\text{d}s=-2\Phi
	\leq 0,
\end{equation*}
and the equality holds if and only if $\kappa=0$. This implies that the total energy decreases with respect to time for mean curvature flow.

\subsection{A discrete mean curvature flow by the Onsager principle} \label{general CSF discrete}
In this section, we aim to numerically solve the dynamic equation \eqref{mean curvature E-L eqaution}. Suppose that the curve $\Gamma(t)$ can be approximated by piecewise linear segments, as illustrated in Fig. \ref{fig:Interpolation curve}. Denote the nodes by
\begin{equation*}
	\x_1(t),\x_2(t),...,\x_n(t),
\end{equation*}
where $\x_i(t)=(x_i^{(1)}(t),x_i^{(2)}(t))^\top\ (1\leq i \leq n)$. The line segments connecting adjacent nodes are denoted by $\Gamma_i(t)=\overline{\x_i(t)\x_{i+1}(t)}\ (1\leq i \leq n-1)$ and $\Gamma_n(t)=\overline{\x_n(t)\x_{1}(t)}$. Therefore, the approximate curve is given by $\Gamma_h(t)=\bigcup\limits_{i=1}^n\Gamma_i(t)$.
\begin{figure}[H]
	\centering
	\includegraphics[width=0.5\textwidth]{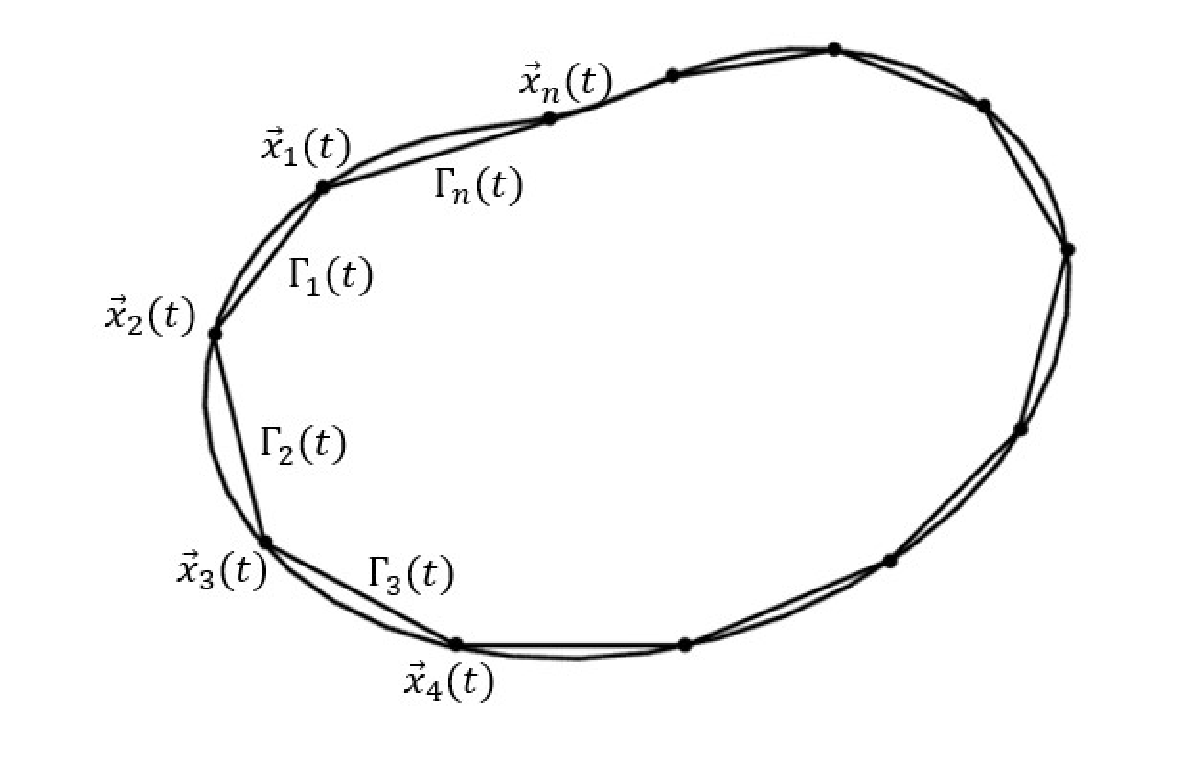}
	\caption{Selection of interpolation points.}
	\label{fig:Interpolation curve}
\end{figure}

We will derive the dynamic equation for $\Gamma_h(t)$ by using the Onsager principle. For a line segment $\Gamma_i(t)\ (1\leq i\leq n)$, the points on the segment can be expressed as 
\begin{equation*}
	\begin{aligned}
		\x_h(t)&=\mu[\x_{i+1}(t)-\x_i(t)]+\x_i(t)
		=(1-\mu)\x_i(t)+\mu\x_{i+1}(t), 
	\end{aligned}  
\end{equation*}
where $0\leq\mu=\frac{s}{|\x_{i+1}(t)-\x_{i}(t)|}\leq1$ is a parameter and $s$ is the arc length parameter of $\Gamma_i(t)$. $\x_{n+1}$ is set to be $\x_{1}$. The velocity of each point $\x_h(t)$ is given by
\begin{equation*}
	\vec{v}_h(t)=(1-\mu)\dot{\x}_i(t)+\mu\dot{\x}_{i+1}(t).
\end{equation*}
The unit tangent vector of $\Gamma_i(t)$ is calculated by
\begin{equation*}
	\vec{\tau}_i=\frac{\x_{i+1}(t)-\x_i(t)}{|\x_{i+1}(t)-\x_i(t)|}.
\end{equation*}
The outward unit normal vector is given by
\begin{equation*}
	\vec{n}_i=\frac{P[\x_{i+1}(t)-\x_i(t)]}{|\x_{i+1}(t)-\x_i(t)|},
\end{equation*}
where the rotation matrix $P=\bigl(\begin{smallmatrix} 0&1\\-1&0 \end{smallmatrix}\bigr)$.

For the discrete curve $\Gamma_h(t)$, the discrete energy is calculated by
\begin{equation*}   
	\mathcal{E}_h=\int_{\Gamma_h(t)}\text{d}s=\sum_{i=1}^n\int_{\Gamma_i(t)}\text{d}s=\sum_{i=1}^n|\Gamma_i(t)|=\sum_{i=1}^n|\x_{i+1}(t)-\x_i(t)|.  
\end{equation*}
The discrete dissipation function is calculated by
\begin{equation*}
	\begin{aligned}
		\Phi_h=\frac{1}{2}\int_{\Gamma_h(t)}|\vec{v}_h|^2\text{d}s.
	\end{aligned}
\end{equation*}
By the Onsager principle, we minimize the discrete Rayleighian with respect to $\dot{\x}_i$, i.e.
\begin{equation}
	\min_{\dot{\x}_i}\mathcal{R}_h:=\Phi_h+\dot{\mathcal{E}}_h.
\end{equation}
The corresponding Euler-Lagrange equation is
\begin{equation}
	\frac{\partial \Phi_h}{\partial \dot{\x}_i}+\frac{\partial \mathcal{E}_h}{\partial \x_i}=0.\label{discrete E-L}
\end{equation}
This leads to a system of ordinary differential equations with respect to the variables $\dot{\x}_i\ (1\leq i \leq n)$.

We can specify the explicit formula for \eqref{discrete E-L} by direct calculations. Denote by $\dot{X}=(\dot{\x}_{1}^\top,\dot{\x}_{2}^\top,...,\dot{\x}_{n}^\top)^\top$ and notice that $\Phi_h$ is a quadratic function with respect to $\dot{\x}_i\ (1\leq i\leq n)$. The equation set \eqref{discrete E-L} can be rewritten as 
\begin{equation}
	A\dot{X}=G,\label{AV=G}
\end{equation}
where the coefficient matrix $A$ is symmetric positive definite and the right-hand term $G=(\vec{g}_1^\top,\vec{g}_2^\top,...,\vec{g}_n^\top)^\top$,
such that $\Phi_h=\frac{1}{2}  \dot{X}^\top A\dot{X}$ and $\vec{g}_i=-\frac{\partial   \mathcal{E}_h}{\partial \vec{x}_i} $. The explicit formulas for $A$ and $G$ are given in Appendix.

We can easily prove the discrete energy dissipative property of the ODE system \eqref{AV=G}.
\begin{theorem} \label{energy decreasing}
	Suppose $\dot{X}=(\dot{\x}_{1}^\top,\dot{\x}_{2}^\top,...,\dot{\x}_{n}^\top)^\top$ is the solution of the system \eqref{AV=G}, then we have
	\begin{equation*}
		\frac{\text{d}\mathcal{E}_h}{\text{d}t}=-2\Phi_h\leq 0,
	\end{equation*}
	where the equality holds if and only if $\dot{\x}_i=0\ (0\leq i \leq n)$.
\end{theorem}
\begin{proof}
	Direct calculations give
	\begin{equation*}
		\frac{\text{d} \mathcal{E}_h}{\text{d} t}=\sum_{i=1}^n\frac{\partial \mathcal{E}_h}{\partial \x_i}\cdot\dot{\x}_i=-\sum_{i=1}^n\vec{g}_i\cdot\dot{\x}_i=-\dot{X}^\top G=-\dot{X}^\top A \dot{X}=-2\Phi_h.
	\end{equation*}
	Here we have used Eq. \eqref{AV=G}. Since the coefficient matrix $A$ is positive definite, we prove the theorem. 
\end{proof}

\subsection{Stabilization for uniform distribution of vertexes} \label{general penalty}
In real simulations, it is necessary to introduce a penalty term to the energy function $\mathcal{E}_h$ in order to avoid vertexes accumulations on the discrete curve. For that purpose, we add a penalty term to the discrete energy function $\mathcal{E}_h$ as follows, 
\begin{equation}
	\begin{aligned}
		\mathcal{E}_h^{\delta}
		&=|\Gamma_h(t)|+\delta\sum_{i=1}^n\left (\frac{|\Gamma_i(t)|}{|\Gamma_{i+1}(t)|}-1\right)^2\\
		&=\sum_{i=1}^n\left [|\x_{i+1}(t)-\x_{i}(t)|+\delta\left (\frac{|\x_{i+1}(t)-\x_{i}(t)|}{|\x_{i+2}(t)-\x_{i+1}(t)|}-1\right )^2\right ].
	\end{aligned}
\end{equation}
In general may take $\delta=\frac{1}{n}$ in order to reduce the impact of the penalty term on $|\Gamma_h(t)|$. By the Onsager principle, we can obtain $\dot{\x}_i$ by minimizing the corresponding Rayleighian, i.e.
\begin{equation}
	\min_{\dot{\x}_i}\mathcal{R}_h^{\delta}:=\Phi_h+\dot{\mathcal{E}}_h^{\delta}.
\end{equation}
The corresponding Euler-Lagrange equation is 
\begin{equation*} 
	\frac{\partial \Phi_h}{\partial \dot{\x}_i}+\frac{\partial \mathcal{E}_h^{\delta}}{\partial \x_i}=0.
\end{equation*}
Similarly, the above equation can be rewritten as 
\begin{equation}
	A\dot{X}=G^{\delta}, \label{AV=G penalty term}
\end{equation}
where $G^{\delta}=((\vec{g}_1^{\delta})^\top,(\vec{g}_2^{\delta})^\top,...,(\vec{g}_n^{\delta})^\top)^\top$ 
with $\vec{g}_i^\delta =-\frac{\partial\mathcal{E}_h^\delta}{\partial \vec{x}_i}$. We refer to the Appendix for the explicit forms of $g_i^{\delta}$.

\subsection{Mean curvature flow in high dimensional spaces}
In the end of the section, we briefly discuss the mean curvature flow in higher dimensional spaces. For a n-dimensional hypersurface $\Gamma(t)$ in $\mathbb{R}^{n+1}$, we can also derive the mean curvature flow equation and its discretization  by  the Onsager variational principle.

First, suppose the dimensionless surface energy in a physical system is given by 
\begin{equation*}
	\mathcal{E}=\int_{\Gamma(t)}\text{d}s=|\Gamma(t)|.
\end{equation*}
Similarly to the curve case, we compute the changing rate of the total energy  by using the Reynolds transport equation as 
\begin{equation} \label{dE/dt 3D}
	\begin{aligned}
		\dot{\mathcal{E}}=\frac{\text{d}}{\text{d}t}\int_{\Gamma(t)}\text{d}s=\int_{\Gamma(t)}H v_{n}\text{d}s, 
	\end{aligned}
\end{equation}
where $v_n=\vec{v}\cdot\vec{n}$ is the normal velocity and   $H = \sum_{j=1}^{n}\kappa_j$  is the mean curvature of $\Gamma(t)$  with $\kappa_1, ..., \kappa_n$ being the principal curvatures \footnote{Notice that the definition  of the mean curvature differs from the common one $H = \frac{1}{n}\sum_{j=1}^{n}\kappa_j$ \cite{deckelnick2005computation}.} .
Suppose  the dissipation function in the system is given by
\begin{equation*}   
	\Phi=\frac{1}{2}\int_{\Gamma(t)} |\vec{v}|^2 \text{d}s.  
\end{equation*}
Then the corresponding Rayleighian is defined as
\begin{equation*}
	\mathcal{R}=\Phi+\dot{\mathcal{E}}=\int_{\Gamma(t)}\left (\frac{1}{2}|\vec{v}|^2+H v_{n}\right )\text{d}s.
\end{equation*}
By the Onsager principle, we minimize the Rayleighian $\mathcal{R}$ with respect to $\vec{v}$, i.e.
\begin{equation*}
	\min_{\vec{v}}\mathcal{R}. \label{minR v_n 3D}
\end{equation*}
The corresponding Euler-Lagrange equation is
\begin{equation}
	\vec{v}=-H\vec{n}. \label{mean curvature E-L eqaution 3D}
\end{equation}
This gives a mean curvature flow equation in high dimensional space.

To discretize Eq. \eqref{mean curvature E-L eqaution 3D}, we can follow the same approach as the  curve case. We first approximate the surface $\Gamma(t)$ by a discrete surface $\Gamma_h(t):= \bigcup_{T\in\mathcal{T}_h}T$ consists of piecewise nondegenerate n-dimensional simplices $T$. Then we derive  a discrete system  for the vertexes of the triangulation by using the Onsager principle similarly as that in the previous subsections. In principle, we could solve the discrete system  to approximate the mean curvature flow for the hypersurface. However, there may be more restrictions on the time discrete schemes to avoid mesh distortion in high dimensional cases. Some properly designed semi-implicit schemes may have nice property as the parametric finite  element method \cite{Barrett2007OnTV}. This will be left for future work.

\section{Volume preserving mean curvature flow} \label{problem 2}
\subsection{Derivation of the dynamic equation by the Onsager principle}
In many applications, the surface energy is minimized while the volume enclosed by the surface is preserved. This can be described by mean curvature flow with volume preservation. In the following, we will derive the geometric equation for such a problem by using the Onsager principle.

We suppose that the area enclosed by a closed curve $\Gamma(t)$ is $S$. The changing rate of the area with respect to time is calculated by
\begin{equation*}
	\frac{\text{d}S}{\text{d}t}=\int_{\Gamma(t)}\vec{v}\cdot\vec{n}\text{d}s=\int_{\Gamma(t)}v_{n}\text{d}s.
\end{equation*}
We assume that both the energy and the dissipative function are the same as in the section \ref{general CSF model}. The Rayleighian $\mathcal{R}$ is defined in \eqref{Rayleighian}. 
When the area $S$ is preserved, the Onsager principle can be stated as follows. We minimize the Rayleighian $\mathcal{R}$ with respect to $\vec{v}$ under the constraint $\frac{\text{d}S}{\text{d}t}=0$, i.e.
\begin{equation}
	\begin{aligned}
		&\min_{\vec{v}}\ \mathcal{R},\\  \label{2minR v_n}
		&s.t.\ \frac{\text{d}S}{\text{d}t}=0.
	\end{aligned}
\end{equation}
Introduce a Lagrangian multiplier $\lambda$. An augmented Lagrangian is defined as 
\begin{equation*}
	\begin{aligned}
		\mathcal{L}(\vec{v}, \lambda) 
		= \mathcal{R} + \lambda\frac{\text{d}S}{\text{d}t}= \int_{\Gamma(t)}\left (\frac{1}{2}|\vec{v}|^2+\kappa v_{n}+\lambda v_{n}\right )\text{d}s.
	\end{aligned}
\end{equation*}
We let
\begin{equation*}
	\left \{ \begin{array}{rcl}
		\frac{\delta \mathcal{L}}{\delta \vec{v}} = 0, \vspace{1ex} \\ 
		\frac{\partial\mathcal{L}}{\partial \lambda} = 0.
	\end{array} \right .
\end{equation*}
By direct calculations, the corresponding Euler-Lagrange equation is given by
\begin{subnumcases} {\label{2E-L}}
	\vec{v} = -(\kappa + \lambda)\vec{n}, \label{2E-L-1}\\
	\int_{\Gamma(t)} v_{n} \text{d}s = 0. \label{2E-L-2}
\end{subnumcases}
The equation can be simplified as follows. We substitute Eq. \eqref{2E-L-1} into \eqref{2E-L-2}. This leads to
\begin{equation*}
	\int_{\Gamma(t)}-(\kappa + \lambda) \text{d}s = 0,
\end{equation*}
which implies that
\begin{equation*}
	\lambda|\Gamma(t)| = -\int_{\Gamma(t)} \kappa \text{d}s.
\end{equation*}
For a simple smooth closed curve $\Gamma(t)$ in the plane, the Gauss-Bonnet formula implies that
\begin{equation}\label{e:Gauss-Bonnet}
	\int_{\Gamma(t)} \kappa \text{d}s = 2\pi.
\end{equation}
This leads to
\begin{equation*}
	\lambda = -\frac{2\pi}{|\Gamma(t)|}.
\end{equation*}
Therefore, we obtain the equation of volume preserving mean curvature flow as follows: 
\begin{equation} \label{2v_n}
	\vec{v} = -\left (\kappa - \frac{2\pi}{|\Gamma(t)|}\right )\vec{n}.
\end{equation}

By substituting Eq. \eqref{2v_n} into \eqref{dE/dt}, we have
\begin{equation*}
		\dot{\mathcal{E}} 
		= \int_{\Gamma(t)}\kappa v_{n}\text{d}s
		= -\int_{\Gamma(t)}\kappa\left (\kappa-\frac{2\pi}{|\Gamma(t)|}\right )\text{d}s
		=-\int_{\Gamma(t)}\left (\kappa-\frac{2\pi}{|\Gamma(t)|}\right )^2\text{d}s
		=-2\Phi\leq 0.	
\end{equation*}
Here we have used Eq. \eqref{e:Gauss-Bonnet} so that
\begin{equation*}
	\int_{\Gamma(t)}\frac{2\pi}{|\Gamma(t)|}\left (\kappa-\frac{2\pi}{|\Gamma(t)|}\right )\text{d}s=\frac{2\pi}{|\Gamma(t)|} \left ( \int_{\Gamma(t)}\kappa\text{d}s
	- \frac{2\pi}{|\Gamma(t)|}\int_{\Gamma(t)}\text{d}s \right ) =0.
\end{equation*}
We can easily see that the energy decreases with respect to time under the volume preserving mean curvature flow.

\subsection{Discretization of the dynamic equation by the Onsager principle}

In this subsection, we will numerically solve the dynamic equation \eqref{2E-L} of $\Gamma(t)$ for the volume preserving mean curvature flow. The discretization of the curve is the same as that in Section \ref{general CSF discrete} (Fig. \ref{fig:Interpolation curve}).

The area enclosed by the approximate curve $\Gamma_h(t)$ is denoted as $S_h$. Then the changing rate of $S_h$ with respect to time is calculated as
\begin{equation*}
	\begin{aligned}
		\frac{\text{d}S_h}{\text{d}t}
		=\int_{\Gamma_h(t)} v_{h, n} \text{d}s
		= \sum_{i=1}^{n}\int_{\Gamma_i(t)} \vec{v}_h\cdot\vec{n}_i\text{d}s.
	\end{aligned}
\end{equation*}
By using the Onsager principle, we minimize the discrete Rayleighian with respect to $\dot{\x}_i$ under the area conservation constraint, i.e.
\begin{equation}
	\begin{aligned}
		&\min_{\dot{\x}_i}\ \mathcal{R}_h:=\Phi_h+\dot{\mathcal{E}}_h,\\
		&s.t.\ \frac{\text{d}S_h}{\text{d}t}=0.
	\end{aligned}
\end{equation}
The corresponding discrete augmented Lagrangian is defined as
\begin{equation*}
	\mathcal{L}_h(\dot{\x}_i, \lambda) = \mathcal{R}_h + \lambda\frac{\text{d}S_h}{\text{d}t}.
\end{equation*}
We let
\begin{equation*}
	\left \{ \begin{array}{rcl}
		\frac{\partial\mathcal{L}_h}{\partial \dot{\x}_i} = 0, \vspace{1ex} \\
		\frac{\partial\mathcal{L}_h}{\partial \lambda} = 0.
	\end{array} \right .
\end{equation*}
The corresponding Euler-Lagrange equation is given by
\begin{equation} \label{2discrete E-L}
	\left \{ \begin{array}{ll}
		\frac{\partial\Phi_h}{\partial\dot{\x}_i} + \lambda\frac{\partial(S_h)_t}{\partial\dot{\x}_i} = -\frac{\partial\mathcal{E}_h}{\partial\x_i}, \vspace{1ex} \\
		\frac{\text{d}S_h}{\text{d}t} = 0. 
	\end{array} \right .
\end{equation}
This is a system of differential-algebraic equations with respect to $\dot{\x}_i\ (1\leq i \leq n)$ and $\lambda$.

Denote by $\dot{X}=(\dot{\x}_{1}^\top,\dot{\x}_{2}^\top,...,\dot{\x}_{n}^\top,\lambda)^\top$ and notice that $\Phi_h$ is a quadratic function with respect to $\dot{\x}_i\ (1\leq i\leq n)$. The system \eqref{2discrete E-L} can be rewritten as
\begin{equation}
	\hat{A}\dot{X}=\hat{G},\label{2AV=G}
\end{equation}
where the coefficient matrix $\hat{A}$ is a symmetric matrix and $\hat{G}=(\vec{g}_{1}^\top,\vec{g}_{2}^\top,...,\vec{g}_{n}^\top,g_{n+1})^\top$. The explicit formulas for $\hat{A}$ and $\hat{G}$ are given in Appendix.

In the following, we will prove the discrete energy decreasing property of the system \eqref{2AV=G}.
\begin{theorem} \label{Th2}
	Suppose $\dot{X}=(\dot{\x}_{1}^\top,\dot{\x}_{2}^\top,...,\dot{\x}_{n}^\top,\lambda)^\top$ is the solution of Eq. \eqref{2AV=G}, then we have
	\begin{equation*}
		\frac{\text{d}\mathcal{E}_h}{\text{d}t}=-2\Phi_h \leq 0,
	\end{equation*}
	where the equality holds if and only if $\dot{\x}_i=0\ (0\leq i \leq n)$.
\end{theorem}
\begin{proof}
	First introduce some notations, $A_0:=\hat{A}(1,... ,2n|1,... ,2n), \dot{X}_0:=\dot{X}(1,...,2n)$. Direct calculations give
	\begin{equation*}
		\begin{aligned}
			\frac{\text{d}\mathcal{E}_h}{\text{d} t}
			&=\sum_{i=1}^n\frac{\partial\mathcal{E}_h}{\partial\x_i}\cdot\dot{\x}_i 
			= -\sum_{i=1}^n\left [\frac{\partial\Phi_h}{\partial\dot{\x}_i} + \lambda\frac{\partial(S_h)_t}{\partial\dot{\x}_i}\right ]\cdot\dot{\x}_i\\ 
			&= -\sum_{i=1}^n\frac{\partial\Phi_h}{\partial\dot{\x}_i}\cdot\dot{\x}_i - \lambda\sum_{i=1}^n\frac{\partial(S_h)_t}{\partial\dot{\x}_i}\cdot\dot{\x}_i 
			= -\dot{X}_0^\top A_0 \dot{X}_0 - \lambda\frac{\text{d}S_h}{\text{d}t}\\
			&= -\dot{X}_0^\top A_0 \dot{X}_0=-2\Phi_h.
		\end{aligned}
	\end{equation*}
	Here we have used Eq. \eqref{2discrete E-L}. It is also easy to know that the matrix $A_0$ is symmetric positive definite, so the theorem is proved.
\end{proof}

In  simulations, we also use the modified energy function $\mathcal{E}_h^{\delta}$ to avoid the mesh degeneracy, as  in section \ref{general penalty}, and the corresponding Euler-Lagrange equation is 
\begin{equation}
	\hat{A}\dot{X}=\hat{G}^{\delta}. \label{2AV=G penalty term}
\end{equation}

\section{Wetting problems} \label{problem 3}
In this section, we further apply the method to wetting problems which can be formulated as a volume preserving mean curvature flow for curves in contact with the substrates.

\subsection{Derivation of the dynamic equation by the Onsager principle}
We consider a two dimensional wetting problem, as illustrated in Fig. \ref{fig:LVS_system}, where the total energy of a liquid-vapor-solid system is composed of three components (c.f. \cite{Xu2016AnET}), i.e.
\begin{equation} \label{LVS energy}
	\tilde{\mathcal{E}}=\gamma_{LV}|\Gamma_{LV}|+\gamma_{SL}|\Gamma_{SL}|+\gamma_{SV}|\Gamma_{SV}|,
\end{equation}
where $\gamma_{LV}$, $\gamma_{SL}$ and $\gamma_{SV}$ represent the energy densities of the liquid-vapor interface $\Gamma_{LV}$, the solid-liquid interface $\Gamma_{SL}$ and the solid-vapor interface $\Gamma_{SV}$ respectively. Assuming that the solid boundary $\Gamma_S=\Gamma_{SL}\bigcup\Gamma_{SV}$ is homogeneous, then both $\gamma_{SL}$ and $\gamma_{SV}$ are constants. In equilibrium the angle between the liquid-vapor interface and the solid interface is determined by Young's equation \cite{YoungIIIAE}: 
\begin{equation*}
	\gamma_{LV}\cos\theta_Y=\gamma_{SV}-\gamma_{SL}.
\end{equation*}
By using the Young's equation, \eqref{LVS energy} can be further simplified to
\begin{equation} \label{calculate energy}
	\begin{aligned}
		\tilde{\mathcal{E}}
		&=\gamma_{LV}|\Gamma_{LV}|-\frac{\gamma_{LV}\cos\theta_Y}{2}|\Gamma_{SL}|+\frac{\gamma_{LV}\cos\theta_Y}{2}(|\Gamma_{S}|-|\Gamma_{SL}|)+\frac{\gamma_{SL}+\gamma_{SV}}{2}|\Gamma_{S}|\\
		&=\gamma_{LV}|\Gamma_{LV}|-\gamma_{LV}\cos\theta_Y|\Gamma_{SL}|+C,
	\end{aligned}
\end{equation}
where $C=\gamma_{SV}|\Gamma_S|$.
\begin{figure}[H]
	\centering
	\includegraphics[width=0.5\textwidth]{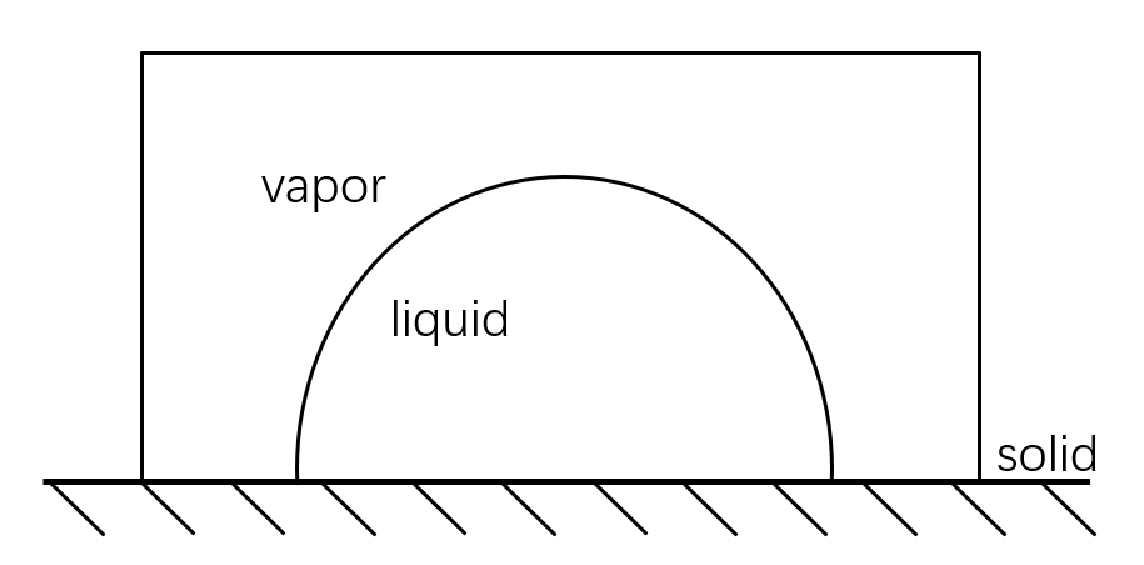}
	\caption{Schematic diagram of a liquid-vapor-solid system.}
	\label{fig:LVS_system}
\end{figure}

Consider a general non-closed curve $\tilde{\Gamma}(t)$, where the left and right endpoints lie on the same horizontal line, as shown in Fig. \ref{fig:Non_closed_curve}. Let $\x(s,t)$ denote the points on the curve $\tilde{\Gamma}(t)$, where $s$ is the arc length parameter. We define the counterclockwise direction as the positive direction of the curve $\tilde{\Gamma}(t)$, and denote its length as $L(t)$. The right endpoint is represented by $\x(0,t)$ and the left endpoint by $\x(L,t)$. The left and right contact angles between the curve $\tilde{\Gamma}(t)$ and the horizontal line are denoted by $\theta_L$ and $\theta_0$, respectively. The unit tangent vector is denoted by $\vec{\tau}$ and the outward unit normal vector by $\vec{n}$. Under the constraint that the area enclosed by $\tilde{\Gamma}(t)$ and the horizontal line remains constant, we aim to derive the evolution equation of the non-closed curve $\tilde{\Gamma}(t)$ under volume preserving mean curvature flow.
\begin{figure}[H]
	\centering
	\includegraphics[width=0.5\textwidth]{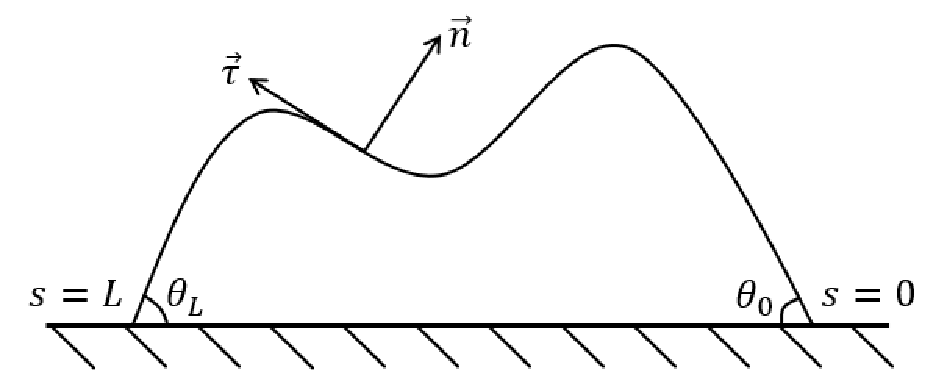}
	\caption{A non-closed curve whose two endpoints always lie on the same horizontal line.}
	\label{fig:Non_closed_curve}
\end{figure}

The total surface energy in \eqref{calculate energy} can be rewritten as
\begin{equation*}
	\tilde{\mathcal{E}}=\gamma[L(t)-| \x(0,t)-\x(L,t)| \cos\theta_Y]+C,
\end{equation*}
where $\gamma=\gamma_{LV}$ represents the surface energy density and $\theta_Y$ denotes the Young's angle. Since the left and right endpoints of $\tilde{\Gamma}(t)$ lie on the bottom line, they possess solely horizontal velocities. Taking the horizontal right direction as the positive direction, the velocities at the left and right endpoints are denoted by $v_L(t)$ and $v_0(t)$, respectively. The changing rate of the length $L(t)$ of $\tilde{\Gamma}(t)$ with respect to time is given by (c.f. Lemma 5.1 in \cite{lu2021efficient})
\begin{equation*}
	\frac{\text{d}L(t)}{\text{d}t}=\int_{\tilde{\Gamma}(t)}\kappa v_n\text{d}s+v_0\cos\theta_0-v_L\cos\theta_L.
\end{equation*}
Therefore, the changing rate of the total energy is calculated as
\begin{equation}
	\begin{aligned} \label{3dE/dt}
		\dot{\tilde{\mathcal{E}}}
		&=\gamma\left [\frac{\text{d}L(t)}{\text{d}t}-\frac{\text{d}}{\text{d}t} |\x(0,t)-\x(L,t)| \cos\theta_Y\right ]\\
		&=\gamma\left [\int_{\tilde{\Gamma}(t)}\kappa v_n\text{d}s+v_0\cos\theta_0-v_L\cos\theta_L-\cos\theta_Y(v_0-v_L)\right ]\\
		&=\gamma\left [\int_{\tilde{\Gamma}(t)}\kappa v_n\text{d}s+v_0(\cos\theta_0-\cos\theta_Y)-v_L(\cos\theta_L-\cos\theta_Y)\right ].
	\end{aligned}
\end{equation}
We suppose that the dissipation function of $\tilde{\Gamma}(t)$ is given by 
\begin{equation} \label{3Phi}
	\tilde{\Phi}=\frac{\xi_0}{2}\int_{\tilde{\Gamma}(t)}|\vec{v}|^2\text{d}s+\frac{\xi_1}{2}v_0^2+\frac{\xi_1}{2}v_L^2,
\end{equation}
where $\xi_0, \xi_1$ are positive friction coefficients. Notice we introduce additional terms $\frac{\xi_1}{2}v_0^2$ and $\frac{\xi_1}{2}v_L^2$ to the dissipation function $\tilde{\Phi}$, which are dissipations due to the contact line frictions. As a result, it characterize the relaxation process of the contact angles from their initial values $\theta_L$ and $\theta_0$ to the equilibrium Young's angle $\theta_Y$. Then the corresponding Rayleighian is given by
\begin{equation*}
	\tilde{\mathcal{R}}=\tilde{\Phi}+\dot{\tilde{\mathcal{E}}}.     
\end{equation*} 
In wetting problems, the volume of a liquid is preserved. The time derivative of the area $\tilde{S}$ enclosed by the curve $\tilde{\Gamma}(t)$ and the bottom line is given by (c.f. Lemma5.1 in \cite{lu2021efficient}) 
\begin{equation*}
	\frac{\text{d}\tilde{S}}{\text{d}t}=\int_{\tilde{\Gamma}(t)}v_n \text{d}s.
\end{equation*}

By using the Onsager principle, we minimize the Rayleighian $\tilde{\mathcal{R}}$ with respect to $\vec{v}$, $v_0$, $v_L$ under the area conservation condition, i.e.
\begin{equation*}
	\begin{aligned}
		&\min_{\vec{v},v_0,v_L}\ \tilde{\mathcal{R}},\\  \label{3minR v_n}
		&s.t.\ \frac{\text{d}\tilde{S}}{\text{d}t}=0.
	\end{aligned}
\end{equation*}
Introduce a Lagrangian multiplier $\lambda$, the corresponding augmented Lagrangian is defined as
\begin{equation*}
	\begin{aligned}
		\tilde{\mathcal{L}}(\vec{v},v_0,v_L,\lambda) 
		= \tilde{\mathcal{R}} + \lambda\frac{\text{d}\tilde{S}}{\text{d}t}
		&= \int_{\tilde{\Gamma}(t)}\left (\frac{\xi_0}{2}|\vec{v}|^2+\gamma\kappa v_{n}+\lambda v_{n}\right )\text{d}s+\frac{\xi_1}{2}v_0^2+\frac{\xi_1}{2}v_L^2\\
		&\quad +\gamma[v_0(\cos\theta_0-\cos\theta_Y)-v_L(\cos\theta_L-\cos\theta_Y)].
	\end{aligned}
\end{equation*}
We let
\begin{equation*}
	\left \{ \begin{array}{rcl}
		\frac{\delta\tilde{\mathcal{L}}}{\delta \vec{v}} = 0, \vspace{1ex} \\
		\frac{\partial\tilde{\mathcal{L}}}{\partial v_0} = 0, \vspace{1ex} \\
		\frac{\partial\tilde{\mathcal{L}}}{\partial v_L} = 0, \vspace{1ex} \\
		\frac{\partial\tilde{\mathcal{L}}}{\partial \lambda} = 0.
	\end{array} \right .
\end{equation*}
This leads to the following Euler-Lagrange equation
\begin{subnumcases} {\label{3E-L}}
	\xi_0\vec{v} = -(\gamma\kappa + \lambda)\vec{n}, \label{3E-L-1}\\
	\xi_1v_0=-\gamma(\cos\theta_0-\cos\theta_Y),\\
	\xi_1v_L=\gamma(\cos\theta_L-\cos\theta_Y),\\
	\int_{\tilde{\Gamma}(t)} v_n \text{d}s = 0. \label{3E-L-4}
\end{subnumcases}
Substitute Eq. \eqref{3E-L-1} into \eqref{3E-L-4}. We get
\begin{equation*}
	\int_{\tilde{\Gamma}(t)}-(\gamma\kappa + \lambda) \text{d}s = 0.
\end{equation*}
This implies that
\begin{equation} \label{3lambda}
	\lambda L(t) = -\gamma\int_{\tilde{\Gamma}(t)} \kappa \text{d}s.
\end{equation}
For the closed curve consisting of $\tilde{\Gamma}(t)$ and the bottom line, applying the Gauss-Bonnet formula yields
\begin{equation*}
	\int_{\tilde{\Gamma}(t)}\kappa\text{d}s + 0\cdot| \x(0,t)-\x(L,t)| +\pi-\theta_0+\pi-\theta_L= 2\pi.
\end{equation*}
This leads to
\begin{equation} \label{3Gauss-Bonnet}
	\int_{\tilde{\Gamma}(t)} \kappa \text{d}s = \theta_0+\theta_L.
\end{equation}
Substituting Eq. \eqref{3Gauss-Bonnet} into \eqref{3lambda}, we have
\begin{equation*}
	\lambda = -\frac{\gamma(\theta_0+\theta_L)}{L(t)}.
\end{equation*}
Therefore, the dynamic equation for a wetting problem can be described by 
\begin{equation} \label{wetting_eq}
	\left \{ \begin{array}{rcl}
		&\vec{v} = -\xi_0^{-1}\gamma\left (\kappa - \frac{\theta_0+\theta_L}{|\tilde{\Gamma}(t)|}\right )\vec{n}, \vspace{1ex} \\
		&v_0=-\xi_1^{-1}\gamma(\cos\theta_0-\cos\theta_Y), \vspace{1ex} \\
		&v_L=\xi_1^{-1}\gamma(\cos\theta_L-\cos\theta_Y),
	\end{array} \right .
\end{equation}
where $\vec{v}$ denotes the velocity of the curve, $v_0$ and $v_L$ are respectively the horizontal velocity of the two endpoints, $\theta_L$ and $\theta_0$ denote the left and right contact angles of the curve with the bottom line, and $|\tilde{\Gamma}(t)|$ denotes the length of the curve. Substituting Eq. \eqref{wetting_eq} into \eqref{3dE/dt}, we get
\begin{equation} \label{3dE/dt<0}
	\begin{aligned}
		\dot{\tilde{\mathcal{E}}} 
		&= -\gamma^2\left [\int_{\tilde{\Gamma}(t)}\xi_0^{-1}\kappa\left (\kappa - \frac{\theta_0+\theta_L}{L(t)}\right )\text{d}s+\xi_1^{-1}(\cos\theta_0-\cos\theta_Y)^2+\xi_1^{-1}(\cos\theta_L-\cos\theta_Y)^2\right ].
	\end{aligned}
\end{equation}
By \eqref{3Gauss-Bonnet}, we can easily derive 
\begin{equation*}
	\begin{aligned}
		&\int_{\tilde{\Gamma}(t)}\kappa\left (\kappa - \frac{\theta_0+\theta_L}{L(t)}\right )\text{d}s=
		\int_{\tilde{\Gamma}(t)}\left (\kappa - \frac{\theta_0+\theta_L}{L(t)}\right )^2\text{d}s.
	\end{aligned}
\end{equation*}
This directly leads to 
\begin{equation*}
	\dot{\tilde{\mathcal{E}}}=-2\tilde{\Phi} \leq 0.
\end{equation*}
This implies that the total energy decreases with respect to time for the system characterized by the dynamic equation \eqref{wetting_eq}.

\subsection{Discretization of the dynamic equation by the Onsager principle} \label{discrete wetting problems}
In this section, we aim to numerically solve the dynamic equation \eqref{3E-L}. Suppose that the curve $\tilde{\Gamma}(t)$ can be approximated by piecewise linear segments, as illustrated in Fig. \ref{fig:3Interpolation curve}.  Denote the nodes by
\begin{equation*}
	\x_1(t),\x_2(t),...,\x_n(t).
\end{equation*}
Where $\x_i(t)=(x_i^{(1)}(t),x_i^{(2)}(t))^\top\ (1\leq i \leq n)$, the right endpoint $\x_1(t)=(x_1^{(1)}(t),0)^\top$ and the left endpoint $\x_n(t)=(x_n^{(1)}(t),0)^\top$. The line segments connecting adjacent nodes are denoted by $\tilde{\Gamma}_i(t)=\overline{\x_i(t)\x_{i+1}(t)}\ (1\leq i \leq n-1)$. Consequently, the approximate curve is given by $\tilde{\Gamma}_h(t)=\bigcup\limits_{i=1}^{n-1}\tilde{\Gamma}_i(t)$.
\begin{figure}[H]
	\centering
	\includegraphics[width=0.5\textwidth]{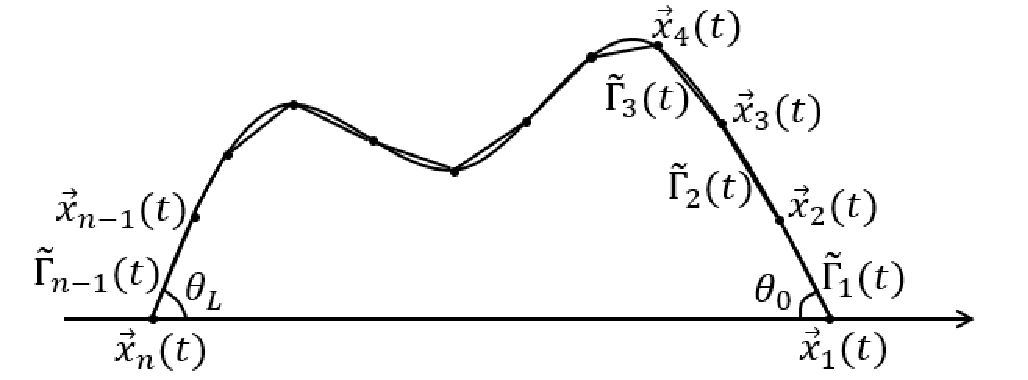}
	\caption{Selection of interpolation points.}
	\label{fig:3Interpolation curve}
\end{figure}

We will derive the dynamic equation for $\tilde{\Gamma}_h(t)$ by using the Onsager principle. Similar to that in Section \ref{general CSF discrete}, for a line segment $\tilde{\Gamma}_i(t)\ (1\leq i\leq n-1)$, the points on the segment can be expressed as 
$
\x_h(t)=(1-\mu)\x_i(t)+\mu\x_{i+1}(t), 
$
where $0\leq\mu=\frac{s}{|\x_{i+1}-\x_{i}|}\leq1$ is a parameter and $s$ is the arc length parameter of $\tilde{\Gamma}_i(t)$.
The velocity of each point $\x_h(t)$ is given by
$
\vec{v}_h(t)=(1-\mu)\dot{\x}_i(t)+\mu\dot{\x}_{i+1}(t).
$
The unit tangent vector of $\tilde{\Gamma}_i(t)$ is calculated by
$
\vec{\tau}_i=\frac{\x_{i+1}(t)-\x_i(t)}{| \x_{i+1}(t)-\x_i(t)|}.
$
The outward unit normal vector is given by
$   \vec{n}_i=\frac{P[\x_{i+1}(t)-\x_i(t)]}{| \x_{i+1}(t)-\x_i(t)|}.$

For the discrete curve $\tilde{\Gamma}_h(t)$, the discrete energy is calculated by
\begin{equation} \label{discrete_energy3}
	\begin{aligned}
		\tilde{\mathcal{E}}_h
		&=\gamma[|\tilde{\Gamma}_h(t)|-| \x_1(t)-\x_n(t)| \cos\theta_Y]+C\\
		&=\gamma\left \{\sum_{i=1}^{n-1}|\x_{i+1}(t)-\x_{i}(t)|-[x_1^{(1)}(t)-x_n^{(1)}(t)]\cos\theta_Y\right \}+C.
	\end{aligned}
\end{equation}
The discrete dissipation function is calculated by
\begin{equation*}
	\begin{aligned}
		\tilde{\Phi}_h
		&=\frac{\xi_0}{2}\int_{\tilde{\Gamma}_h(t)}|\vec{v}_{h}|^2\text{d}s+\frac{\xi_1}{2}v_{h,0}^2+\frac{\xi_1}{2}v_{h, L}^2\\
		&=\frac{\xi_0}{2}\sum_{i=1}^{n-1}\int_{\tilde{\Gamma}_i(t)}[(1-\mu)\dot{\x}_i(t)+\mu\dot{\x}_{i+1}(t)]^2\text{d}s+\frac{\xi_1}{2}[\dot{x}_1^{(1)}(t)]^2+\frac{\xi_1}{2}[\dot{x}_n^{(1)}(t)]^2.
	\end{aligned}
\end{equation*}
Denote by $\tilde{S}_h$ the area enclosed by the discrete curve $\tilde{\Gamma}_h(t)$ and the bottom line. Then the changing rate of $\tilde{S}_h$ with respect to time is calculated as
\begin{equation*}
	\begin{aligned}
		\frac{\text{d}\tilde{S}_h}{\text{d}t}
		&=\int_{\tilde{\Gamma}_h(t)}v_{h,n} \text{d}s
		=\sum_{i=1}^{n-1}\int_{\tilde{\Gamma}_i(t)}\vec{v}_h(t)\cdot\vec{n}_i\text{d}s. 
	\end{aligned}
\end{equation*}
Taking the horizontal right direction as the positive direction, here we denote the normal velocity at the two endpoints by
\begin{equation*}
	\dot{\x}_1(t)\cdot\vec{n}_1=\dot{x}_1^{(1)}(t)\sin\theta_0,\quad
	\dot{\x}_n(t)\cdot\vec{n}_{n-1}=-\dot{x}_n^{(1)}(t)\sin\theta_L,
\end{equation*}
where
\begin{equation*}
	\sin\theta_0=\frac{x_2^{(2)}(t)}{| \x_2(t)-\x_1(t)|},\quad \sin\theta_L=\frac{x_{n-1}^{(2)}(t)}{| \x_n(t)-\x_{n-1}(t)|}.
\end{equation*}
By using the Onsager principle, we minimize the discrete Rayleighian with respect to $\dot{\x}_i (2\leq i \leq n-1), \dot{x}_1^{(1)}, \dot{x}_n^{(1)}$ under the area conservation constraint, i.e.
\begin{equation}
	\begin{aligned}
		&\min_{\dot{\x}_i, \dot{x}_1^{(1)}, \dot{x}_n^{(1)}}\ \tilde{\mathcal{R}}_h:=\tilde{\Phi}_h+\dot{\tilde{\mathcal{E}}}_h,\\
		&s.t.\ \frac{\text{d}\tilde{S}_h}{\text{d}t}=0.
	\end{aligned}
\end{equation}
The corresponding discrete augmented Lagrangian is defined as
\begin{equation*}
	\tilde{\mathcal{L}}_h(\dot{\x}_i, \dot{x}_1^{(1)}, \dot{x}_n^{(1)}, \lambda) = \tilde{\mathcal{R}}_h + \lambda\frac{\text{d}\tilde{S}_h}{\text{d}t}.
\end{equation*}
We set
\begin{equation*}
	\left \{ \begin{array}{rcl}
		\frac{\partial\tilde{\mathcal{L}}_h}{\partial \dot{\x}_i} = 0, \vspace{1ex} \\
		\frac{\partial\tilde{\mathcal{L}}_h}{\partial \dot{x}_1^{(1)}} = 0, \vspace{1ex} \\
		\frac{\partial\tilde{\mathcal{L}}_h}{\partial \dot{x}_n^{(1)}} = 0, \vspace{1ex} \\
		\frac{\partial\tilde{\mathcal{L}}_h}{\partial \lambda} = 0.
	\end{array} \right .
\end{equation*}
The corresponding Euler-Lagrange equation is given by
\begin{equation} \label{3discrete E-L}
	\left\{\begin{array}{ll}
		\frac{\partial\tilde{\Phi}_h}{\partial\dot{\x}_i} + \lambda\frac{\partial(\tilde{S}_h)_t}{\partial\dot{\x}_i} = -\frac{\partial\tilde{\mathcal{E}}_h}{\partial\x_i}, \quad 2\leq i\leq n-1, \vspace{1ex} \\
		\frac{\partial\tilde{\Phi}_h}{\partial\dot{x}_1^{(1)}} + \lambda\frac{\partial(\tilde{S}_h)_t}{\partial\dot{x}_1^{(1)}} = -\frac{\partial\tilde{\mathcal{E}}_h}{\partial x_1^{(1)}}, \vspace{1ex} \\
		\frac{\partial\tilde{\Phi}_h}{\partial\dot{x}_n^{(1)}} + \lambda\frac{\partial(\tilde{S}_h)_t}{\partial\dot{x}_n^{(1)}} = -\frac{\partial\tilde{\mathcal{E}}_h}{\partial x_n^{(1)}}, \vspace{1ex} \\
		\frac{\text{d}\tilde{S}_h}{\text{d}t} = 0.
	\end{array} 
	\right.
\end{equation}
This is a system of differential-algebraic equations with respect to $\dot{\x}_i\ (2\leq i \leq n-1), \dot{x}_1^{(1)}, \dot{x}_n^{(1)}$ and $\lambda$.

For convenience of representation, we add two equations $\dot{x}_1^{(2)}=0, \dot{x}_n^{(2)}=0$ to \eqref{3discrete E-L} and denote $\dot{X}=(\dot{\x}_{1}^\top,\dot{\x}_{2}^\top,...,\dot{\x}_{n}^\top,\lambda)^\top$. Observe that $\tilde{\Phi}_h$ is a quadratic function with respect to $\dot{\x}_i\ (1\leq i\leq n)$. Then the system \eqref{3discrete E-L} can be rewritten as 
\begin{equation}
	\tilde{A}\dot{X}=\tilde{G},\label{3AV=G}
\end{equation}
where $\tilde{A}$ is a symmetric matrix and $\tilde{G}=(\tilde{\vec{g}}_1^\top,\tilde{\vec{g}}_2^\top,...,\tilde{\vec{g}}_n^\top,\tilde{g}_{n+1})^\top$. The explicit formulas for $\tilde{A}$ and $\tilde{G}$ are given in Appendix.

We can easily prove the discrete energy dissipative property of the ODE system \eqref{3AV=G}.
\begin{theorem}
	Suppose $\dot{X}=(\dot{\x}_{1}^\top,\dot{\x}_{2}^\top,...,\dot{\x}_{n}^\top,\lambda)^\top$ is the solution of Eq. \eqref{3AV=G}, then we have
	\begin{equation*}
		\frac{\text{d}\tilde{\mathcal{E}}_h}{\text{d}t}=-2\tilde{\Phi}_h \leq 0,
	\end{equation*}
	where the equality holds if and only if $\dot{\x}_i=0\ (0\leq i \leq n)$.
\end{theorem}
The proof is similar to that of Theorem \ref{Th2} and we skip it for simplicity in presentation.

As in the section \ref{general penalty}, we  add a penalty term to the energy function $\tilde{\mathcal{E}}_h$ in order to make the nodes on the discrete curve uniformly distributed. The modified energy function 
is denoted by
\begin{equation*}
	\tilde{\mathcal{E}}_h^{\delta}=\tilde{\mathcal{E}}_h+\delta\sum_{i=1}^{n-2}\left (\frac{|\tilde{\Gamma}_i(t)|}{|\tilde{\Gamma}_{i+1}(t)|}-1\right )^2.
\end{equation*}
We generally take $\delta=\frac{1}{n-2}$ in order to reduce the impact of the penalty term on $\tilde{\mathcal{E}}_h$. The corresponding Euler-Lagrange equation is 
\begin{equation}\label{3AV=G penalty term}
	\tilde{A}\dot{X}=\tilde{G}^{\delta}.
\end{equation}

\section{Numerical experiments}
We will present some experimental simulation results in this section. In simulations, we solve this problems using the improved Euler's method as follows
\begin{equation} \label{improved Euler}
	\left \{ \begin{array}{rcl}
		\text{predictor}:{\Bar{\Bar X}}^{k+1}&=&X^k+\dot{X}^k\Delta t, \vspace{1ex} \\
		\text{corrector}:X^{k+1}&=&X^k+\frac{(\dot{X}^k+\dot{{\Bar{\Bar X}}}^{k+1})}{2}\Delta t.
	\end{array} \right .
\end{equation}
Where $X^k$ represents the position of the curve at time $t_k$. For Eq. \eqref{AV=G penalty term}, $\dot{X}^k=A^{-1}(X^k)G^\delta(X^k)$ denotes the time derivative of $X^k$, and $\dot{\Bar{\Bar X}}^{k+1}=A^{-1}(\Bar{\Bar X}^{k+1})G^\delta(\Bar{\Bar X}^{k+1})$. For Eqs. \eqref{2AV=G penalty term} and \eqref{3AV=G penalty term}, we use $\hat{A}$, $\hat{G}^{\delta}$ and $\tilde{A}$, $\tilde{G}^{\delta}$ respectively.

\subsection{Mean curvature flow}
\subsubsection{The evolution of a circular curve} \label{experi1_circle}
We consider the evolution of a circle under mean curvature flow. Notice that the curve will always be circular in this case. We can assume that
\begin{equation}
	\Gamma(t):=\{\x(\alpha, t)=\rho(t)(\cos\alpha,\sin\alpha)^\top,\quad 0\leq \alpha \leq 2\pi\},\label{x(t,theta)}
\end{equation}
where $\rho(t)$ denotes the radius of the circle at time $t$, and it satisfies the initial condition $\rho(0)=R_0$. By substituting Eq. \eqref{x(t,theta)} into Eq. \eqref{mean curvature E-L eqaution}, we have
\begin{equation}
	\frac{\text{d}\rho}{\text{d}t}=-\frac{1}{\rho(t)}.\label{circle}
\end{equation}
Solving Eq. \eqref{circle}, we can obtain the evolution equation of the curve $\Gamma(t)$
\begin{equation} \label{exact solution}
	\x(\alpha, t)=\sqrt{R_0^2-2t}(\cos\alpha,\sin\alpha)^\top,\quad 0\leq \alpha \leq 2\pi.
\end{equation}
Hence, under mean curvature flow, a circle of initial radius $R_0$ gradually shrinks to a circle of radius $\sqrt{R_0^2-2t}$ at time $t$, and eventually collapses to a point at time $t=\frac{R_0^2}{2}$.

In our numerical simulations, the initial curve is taken as the unit circle. The circle is discretized uniformly with $\x_i=(\cos\frac{2\pi}{n}i, \sin\frac{2\pi}{n}i)^\top$. The shape of the curve at some time steps are shown in Fig. \ref{fig:circle}, while the change of the discrete energy function $\mathcal{E}_h(t)$ over time is shown in Fig. \ref{fig: circle energy}. We can see that discrete curves are circular and the discrete energy decreases gradually.  
\begin{figure}[h]
	\centering
	\subfigure[t=0]{
		\includegraphics[width=1.5in]{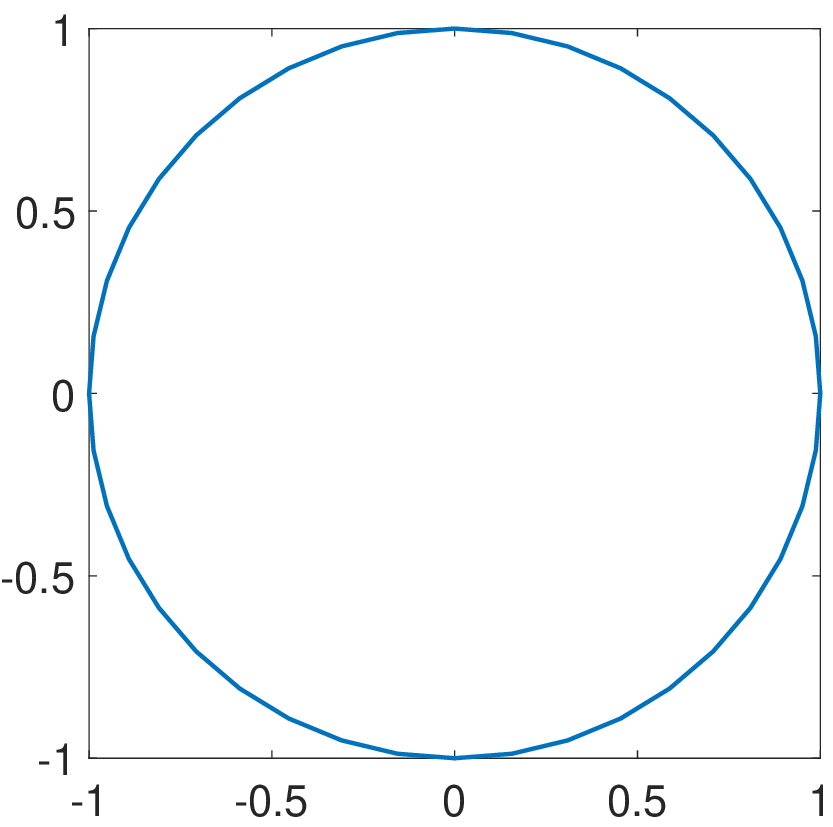}
	}
	\subfigure[t=0.1]{
		\includegraphics[width=1.5in]{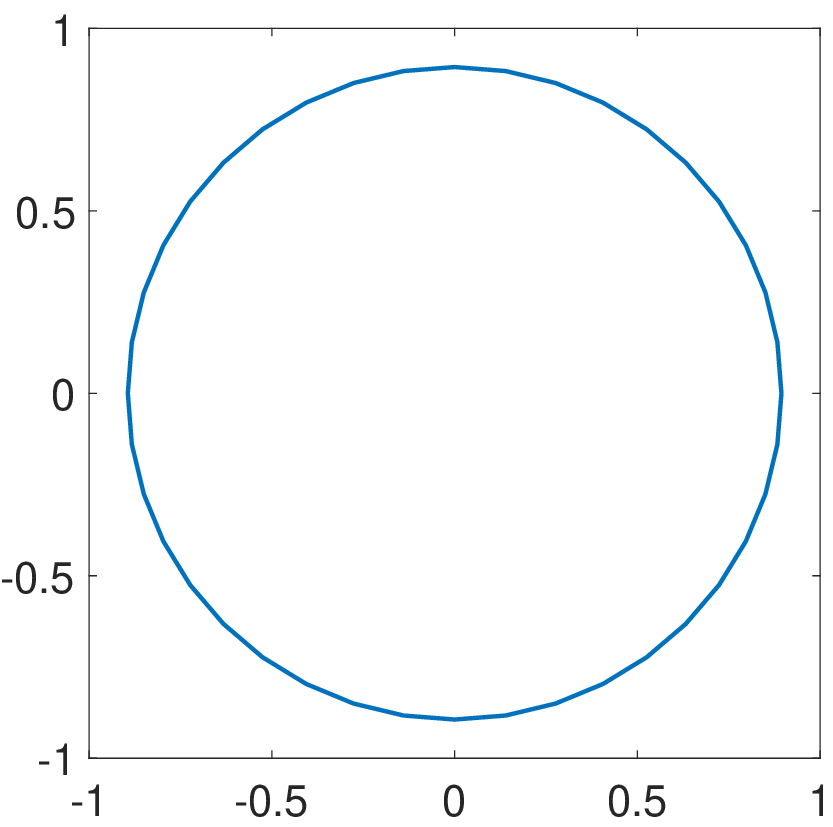}
	}
	\subfigure[t=0.2]{
		\includegraphics[width=1.5in]{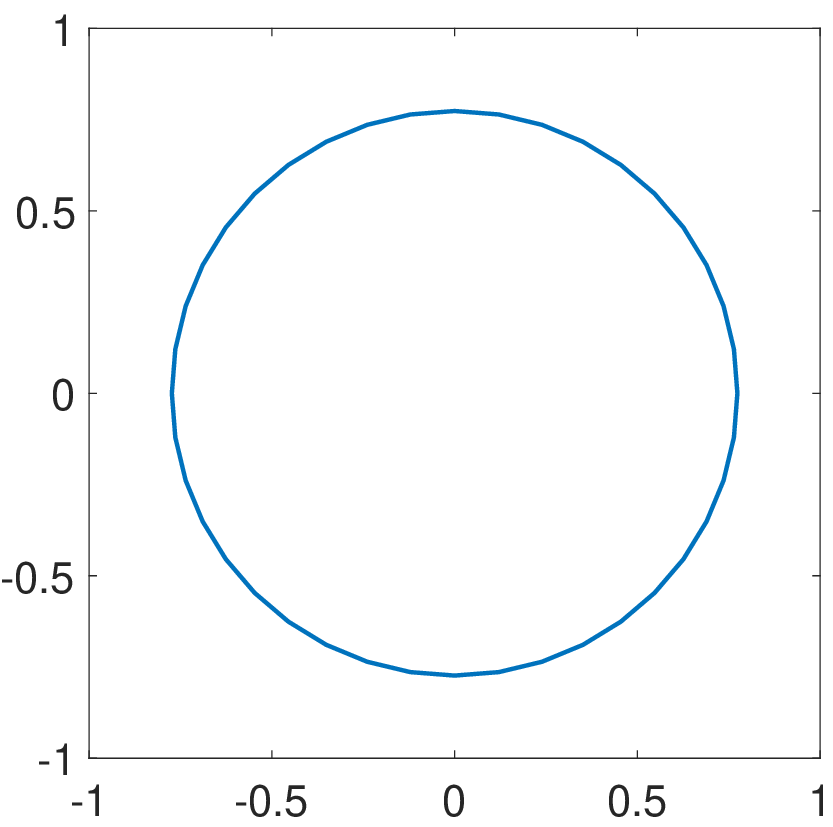}
	}
	
	\subfigure[t=0.3]{
		\includegraphics[width=1.5in]{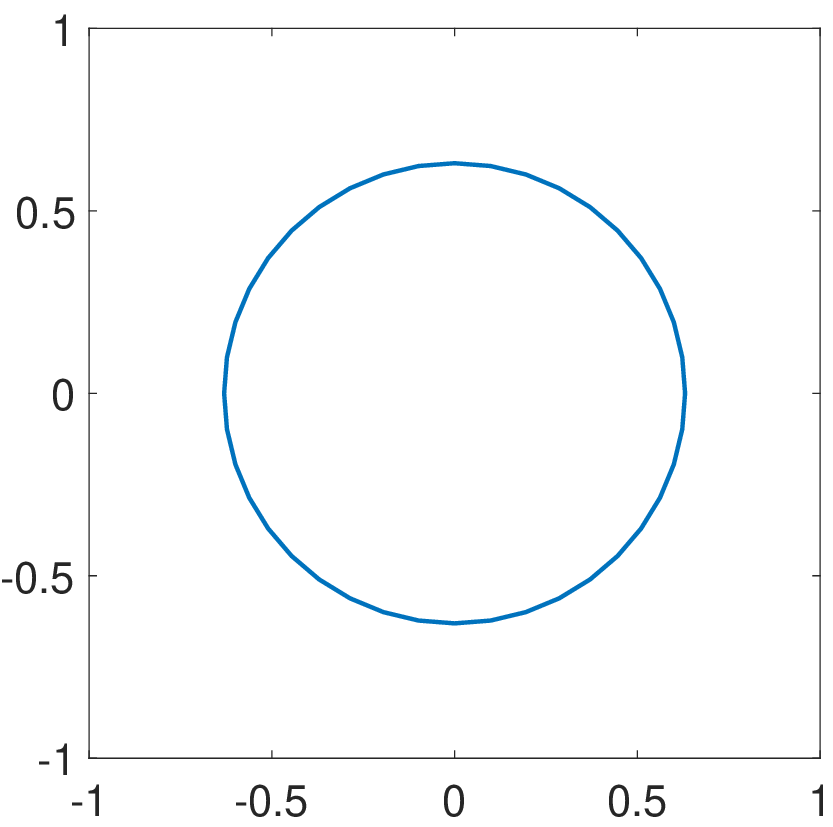}
	}
	\subfigure[t=0.4]{
		\includegraphics[width=1.5in]{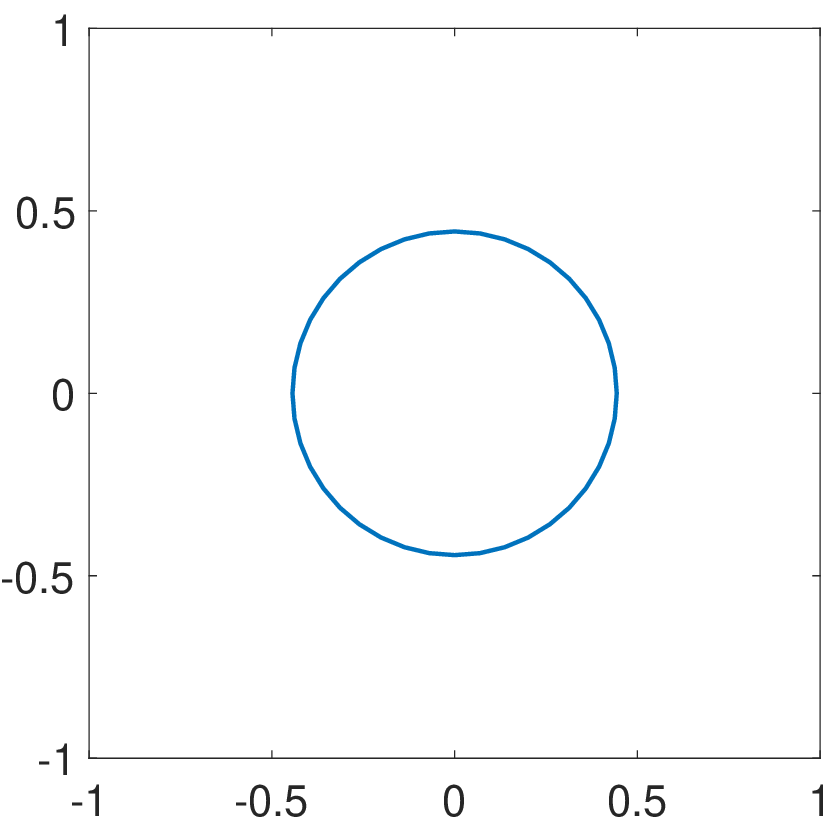}
	}
	\subfigure[t=0.488]{
		\includegraphics[width=1.5in]{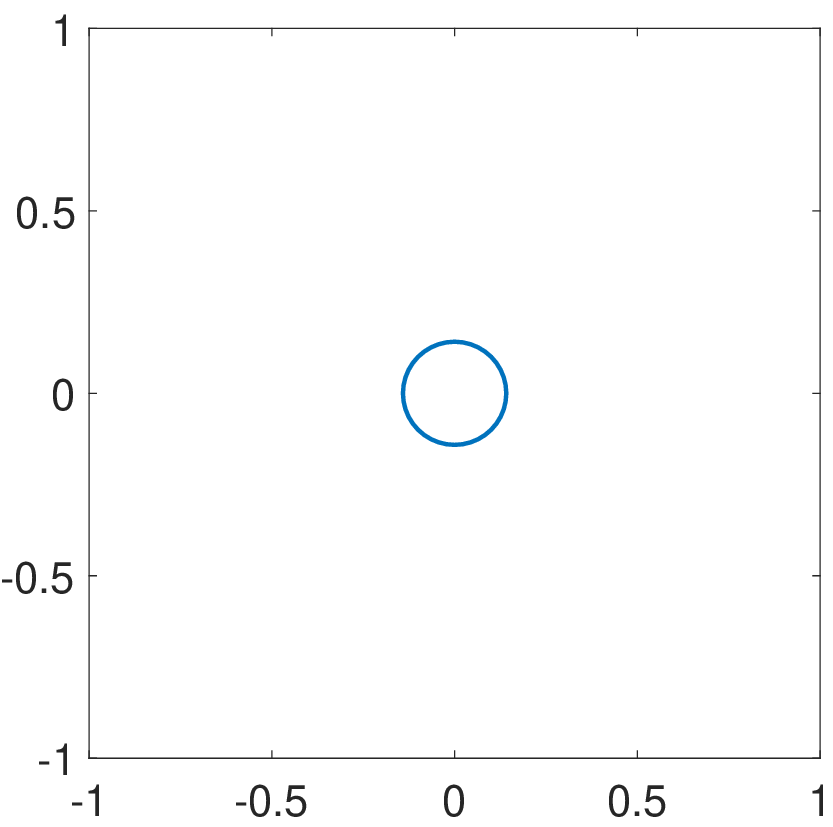}
	}
	\caption{Evolution of a unit circle under mean curvature flow (n=40).}
	\label{fig:circle}
\end{figure}

\begin{figure}[h]
	\centering
	\includegraphics[width=0.5\textwidth]{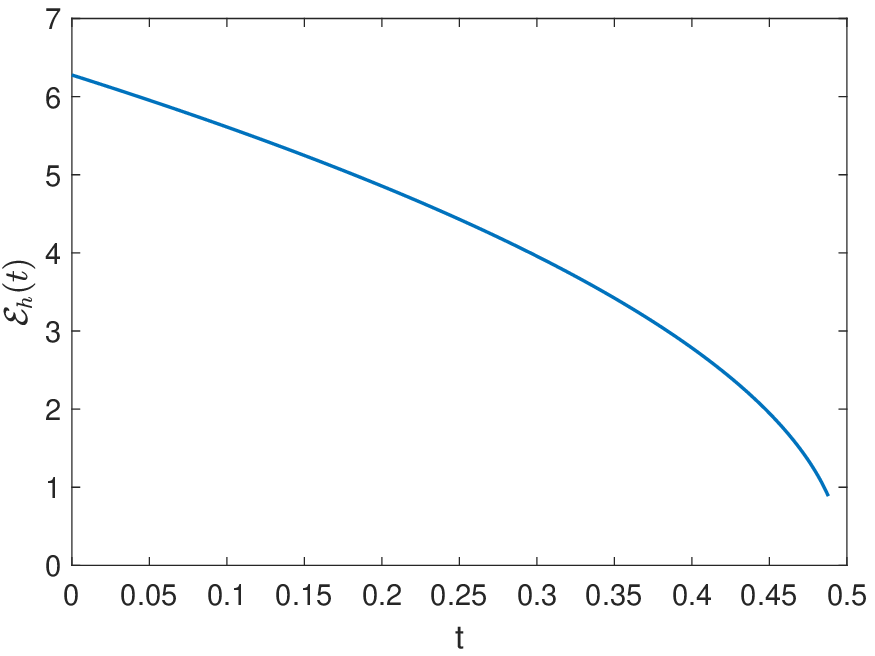}
	\caption{The change of the discrete energy $\mathcal{E}_h(t)$ with respect to time.}
	\label{fig: circle energy}
\end{figure} 

Since the explicit solution is known in this example (as shown in \eqref{exact solution}), we can compute the errors of the numerical solutions. We denote the point on the discrete curve $\Gamma_h(t)$ by $\x_h(\alpha, t)$. We define the error as
\begin{equation} \label{err1}
	err_1:=\max_{\x_h\in \Gamma_h(t)} \text{dist}(\x_h, \Gamma(t)).
\end{equation}
The convergence order, which characterizes the rate of convergence, is calculated as follows
\begin{equation}\label{order}
	order:=\frac{\ln( err_1^{(k)}/err_1^{(k+1)} )}{\ln(n_{k+1}/n_k)}.
\end{equation}

We set $T=0.2$ and choose a time step size of $\Delta t=0.00025$, which is small enough so that the error with respect to time discretization is of higher order. Taking the number of nodes as $n=5,10,20,40,80$, the corresponding numerical errors and convergence order are presented in Table \ref{err order}. We could see the method has second order convergent in distance norm with respect to the spacial mesh size.
\begin{table}[htbp]   
	\begin{center}   
		\caption{Error and convergence order of the numerical method in Section \ref{general CSF discrete}.}  
		\label{err order} 
		\begin{tabular}{c|c c c c c}   
			\hline \textbf{\em n} & 5 & 10 & 20 & 40 & 80 \\   
			\hline \textbf{\em $err_1$} & 8.1561e-02 & 1.7758e-02 & 4.2941e-03 & 1.0647e-03 & 2.6563e-04  \\ 
			\textbf{\em $Order$} & - & 2.20 & 2.05 & 2.01 & 2.00  \\      
			\hline    
		\end{tabular}   
	\end{center}   
\end{table}

\subsubsection{The evolution of a flower-shaped curve} \label{flower1}
We consider the evolution of a flower-shaped curve under mean curvature flow. The parametric equation for the initial flower-shaped curve is given by:
\begin{equation} \label{flower equation}
	\x(\alpha) = (2 - 2^{\sin(5\alpha)})(\cos\alpha, \sin\alpha)^\top,\quad 0\leq \alpha \leq 2\pi.
\end{equation}
We discretize the curve uniformly in $\alpha$ by setting $\x_i=\x(\alpha_i)$ with $\alpha_i=\frac{2\pi}{n}i$. The evolution of the curve is plotted in Fig. \ref{fig:flower}, and the change of the discrete energy function $\mathcal{E}_h(t)$ is illustrated in Fig. \ref{fig: flower energy}. 
We can see that the curve gradually changes to a circular shape and also shrinks in length.
We can also see how the discrete energy function gradually decreases as the curve evolves under mean curvature flow. 
\begin{figure}[h]
	\centering
	\subfigure[t=0]{
		\includegraphics[width=1.3in]{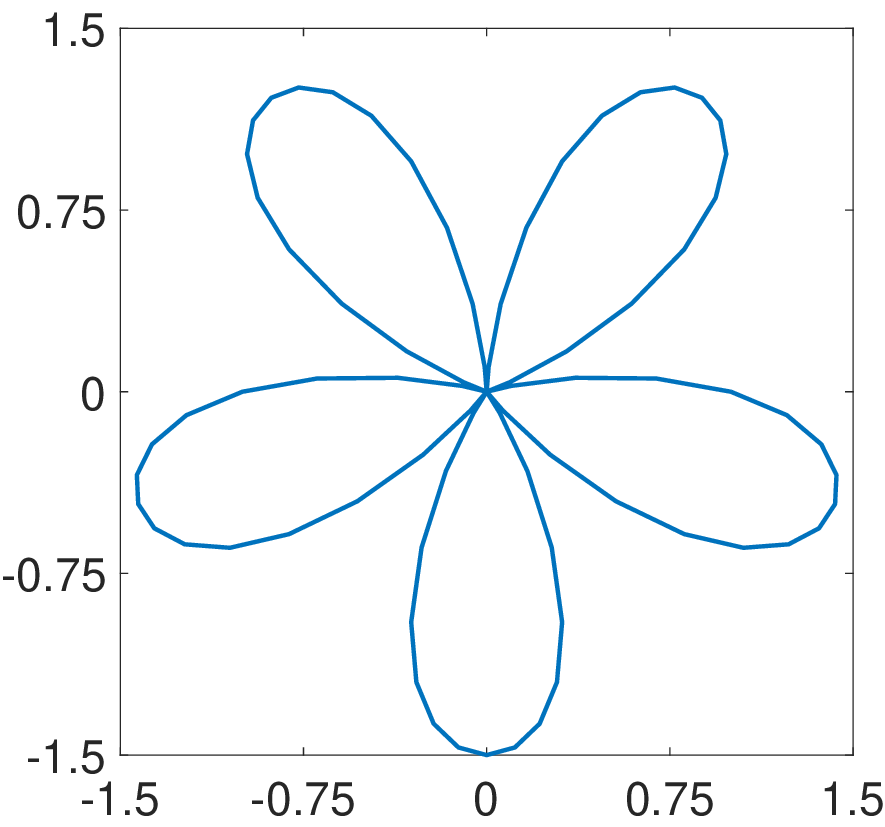}
	}
	\subfigure[t=0.03]{
		\includegraphics[width=1.3in]{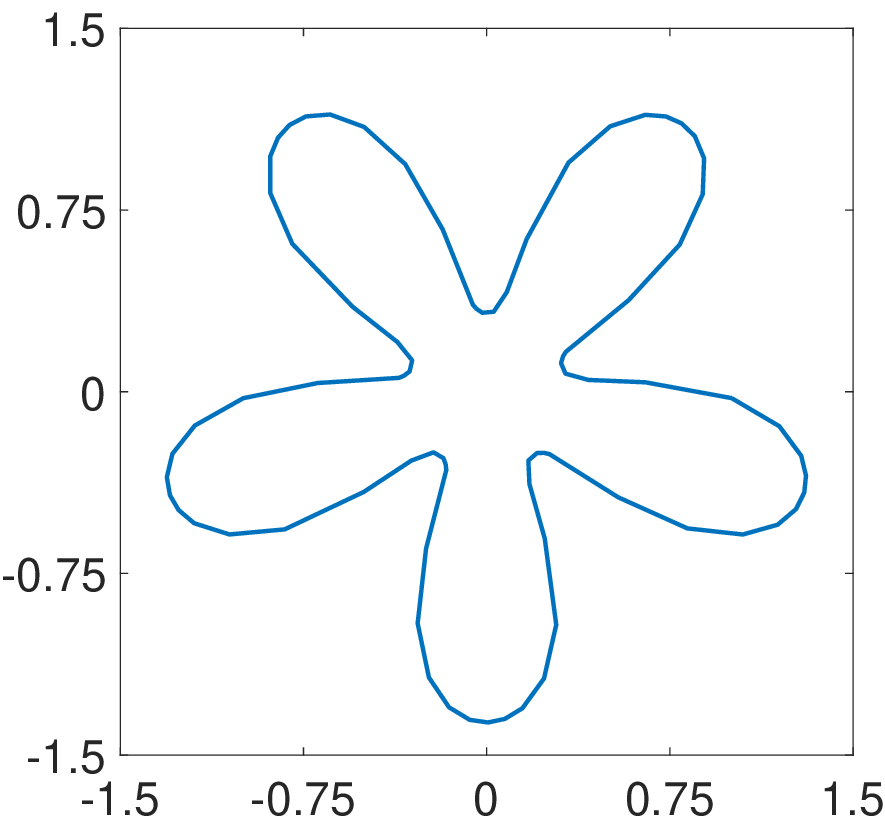}
	}
	\subfigure[t=0.06]{
		\includegraphics[width=1.3in]{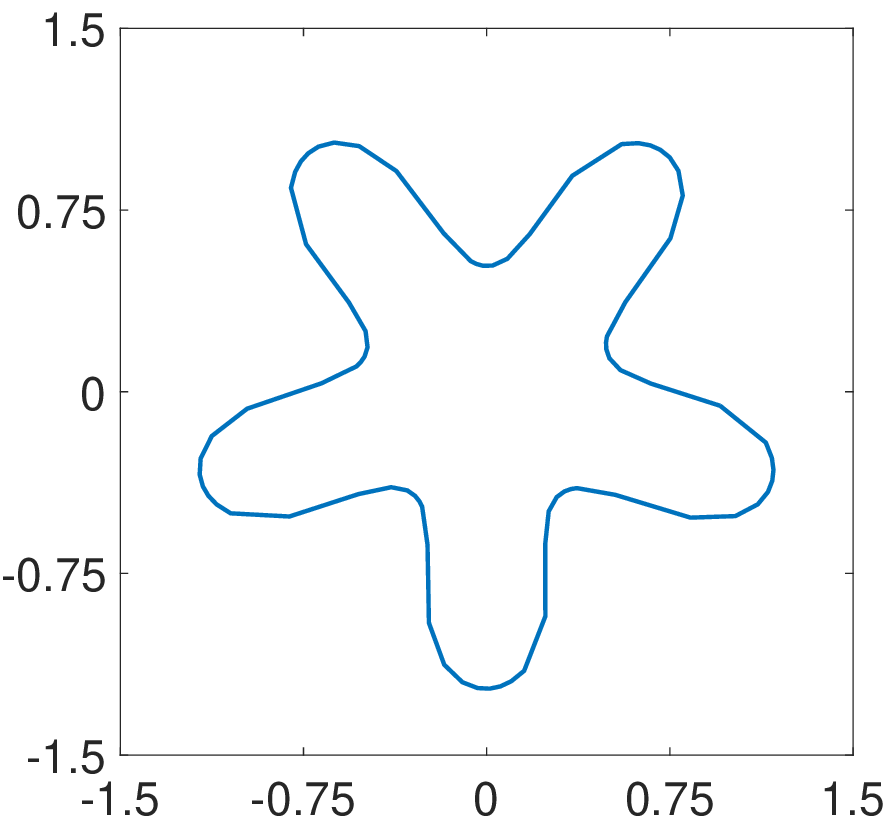}
	}
	
	\subfigure[t=0.09]{
		\includegraphics[width=1.3in]{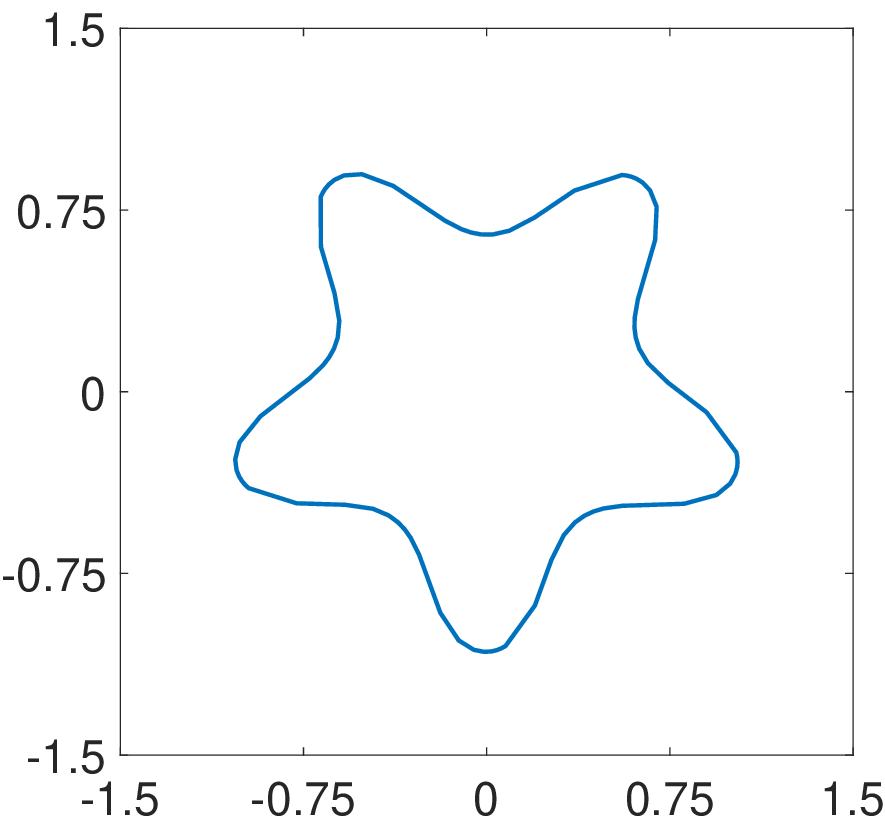}
	}
	\subfigure[t=0.12]{
		\includegraphics[width=1.3in]{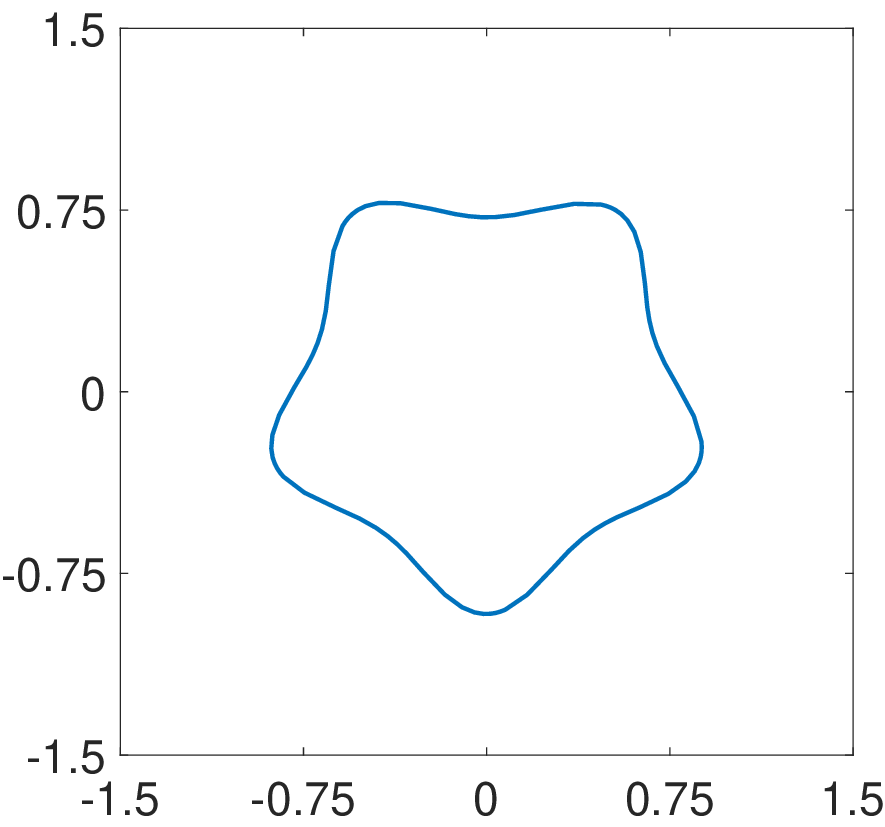}
	}
	\subfigure[t=0.15]{
		\includegraphics[width=1.3in]{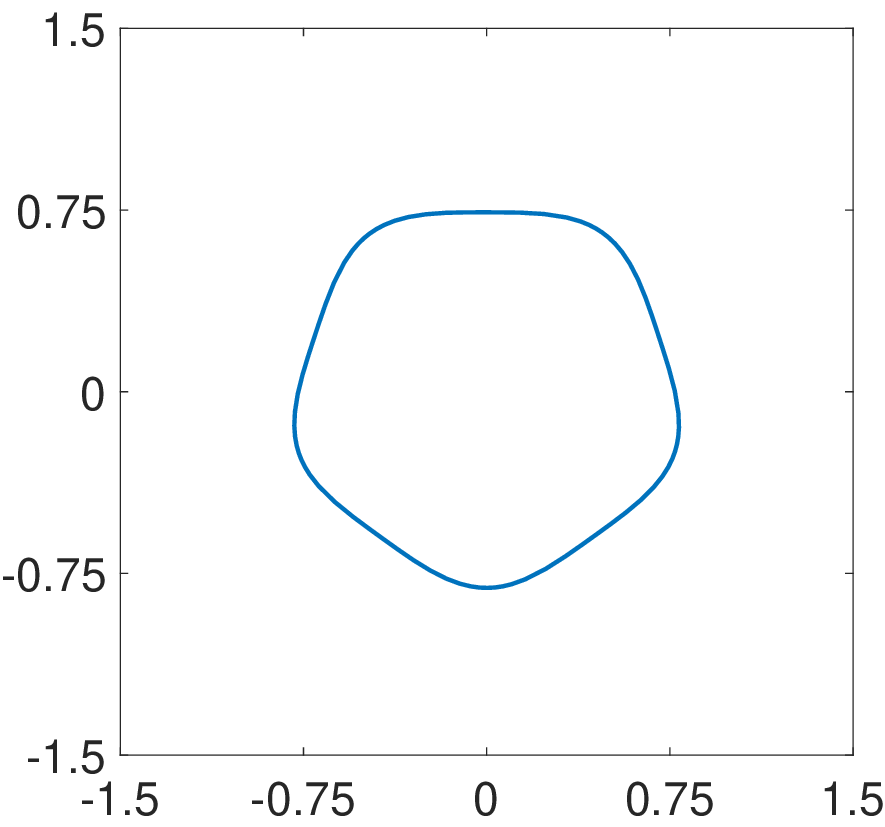}
	}
	
	\subfigure[t=0.18]{
		\includegraphics[width=1.3in]{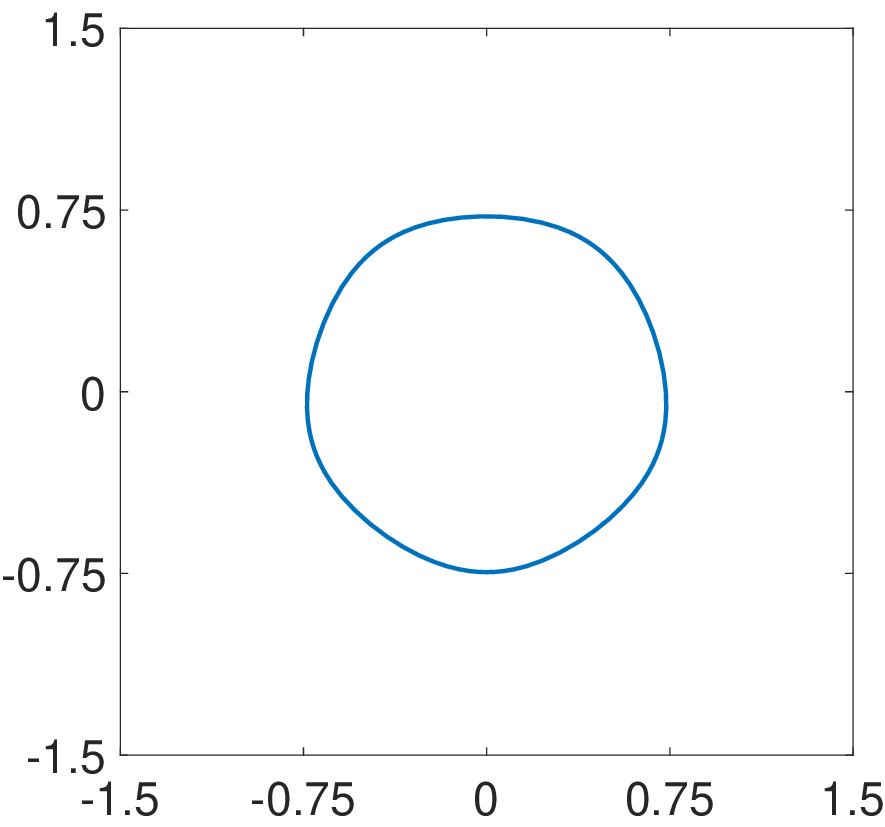}
	}
	\subfigure[t=0.3]{
		\includegraphics[width=1.3in]{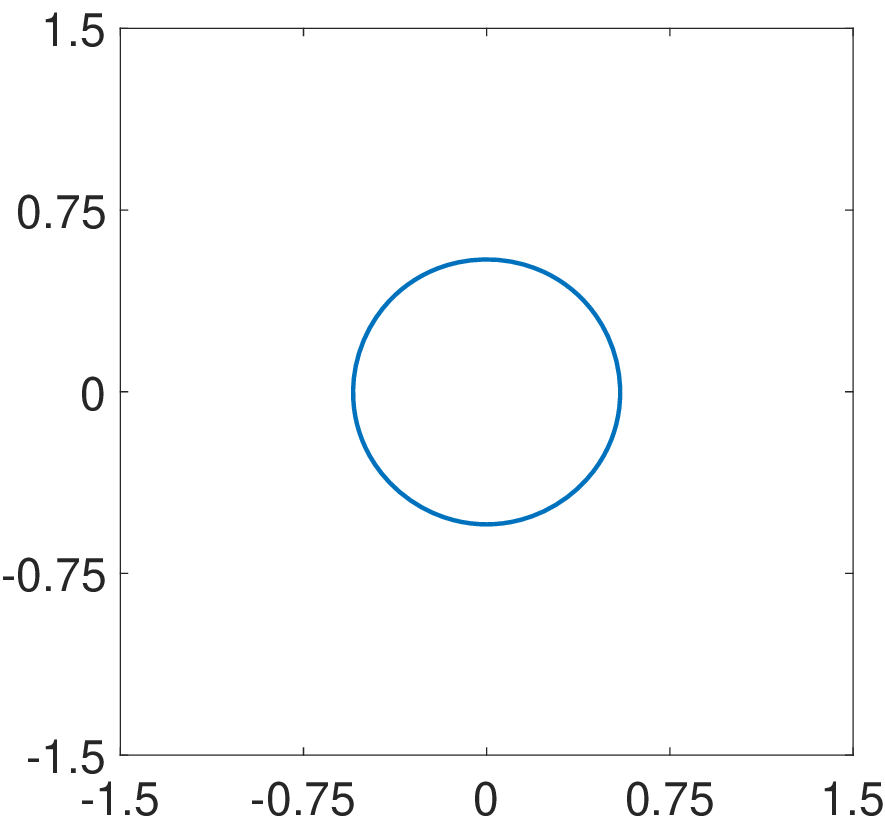}
	}
	\subfigure[t=0.41]{
		\includegraphics[width=1.3in]{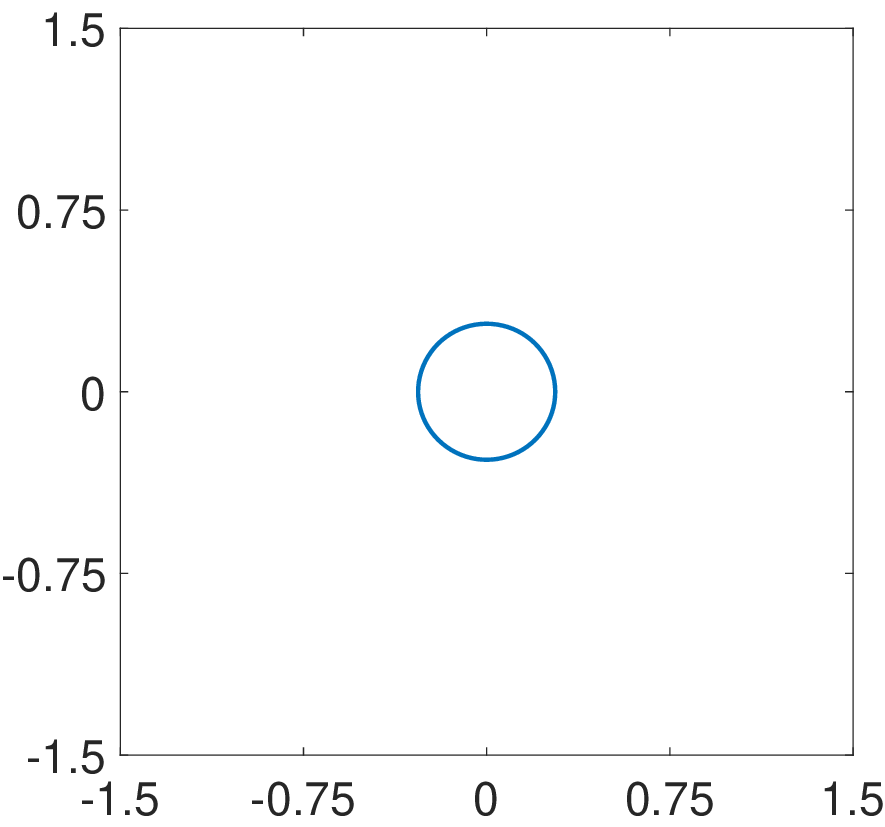}
	}
	\caption{Evolution of a flower-shaped curve under mean curvature flow (n=80).}
	\label{fig:flower}
\end{figure}
\begin{figure}[h]
	\centering
	\includegraphics[width=0.5\textwidth]{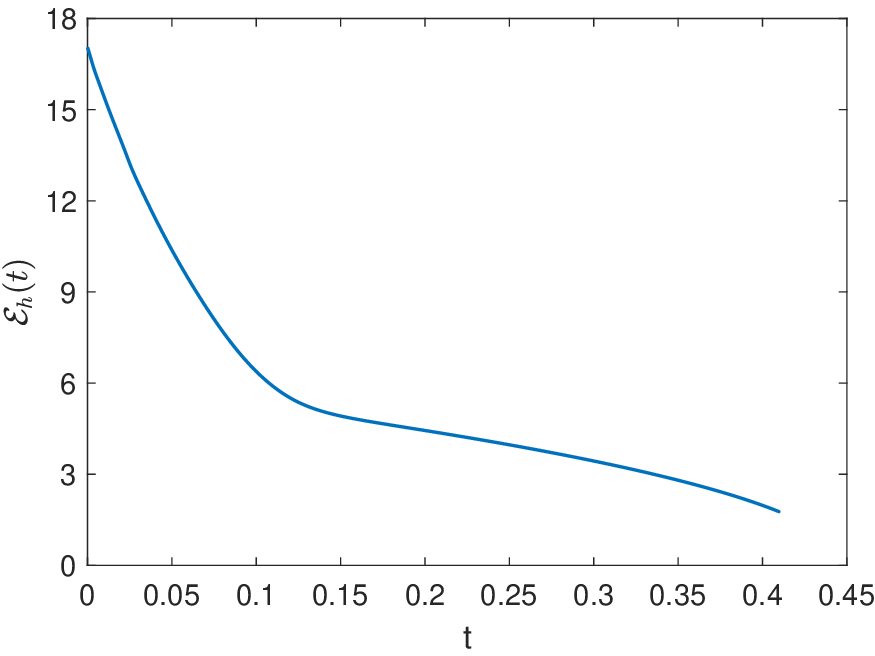}
	\caption{The change of the discrete energy $\mathcal{E}_h(t)$ with respect to time.}
	\label{fig: flower energy}
\end{figure}

\subsubsection{Effect of the penalty term}
In section \ref{general penalty}, we  have added a penalty term into the discrete energy function $\mathcal{E}_h$ to avoid mesh degeneracy. Here, we will show the effect of the penalty. Define a mesh ratio indicator (MRI) $\Psi(t)$ of the discrete curve $\Gamma_h(t)$ (as in \cite{zhao2021energy}) that	
\begin{equation*}
	\Psi(t) = \frac{\max\limits_{1\leq i\leq n} |\x_{i+1} - \x_{i}|}{\min\limits_{1\leq i\leq n} |\x_{i+1} - \x_{i}|}.
\end{equation*}

For the circular curve case, the change of the MRI function $\Psi(t)$ is shown in Fig. \ref{fig:penalty_circle}. Notably, the MRI function $\Psi(t)$ remains constant $1$ (i.e. the mesh maintain a uniform distribution) for both situations with or without the penalty term. 
In this case, the penalty term is not necessary. 
However, for the flower-shaped curve case, the situation is different, as shown in Fig. \ref{fig:penalty_flower}.
It is observed that without the penalty term, $\Psi(t)$ changes drastically over time, indicating very poor mesh property. In comparison, if we include the penalty term, the change of  $\Psi(t)$ is mild and  goes to $1$ gradually. 
Therefore, adding a penalty term in the energy significantly improves the mesh quality during the evolution of a flower-shaped curve under mean curvature flow. We also test the effect of the the penalty term for the numerical experiments in next subsections. The observations are similar.  The penalty term seems not necessary for smooth curves like circles while it helps to improve the mesh quality for more singular curves.

\begin{figure}[h]
	\centering
	\subfigure[$\mathcal{E}_h(t)$ does not include the penalty term.]{
		\includegraphics[width=2.8in]{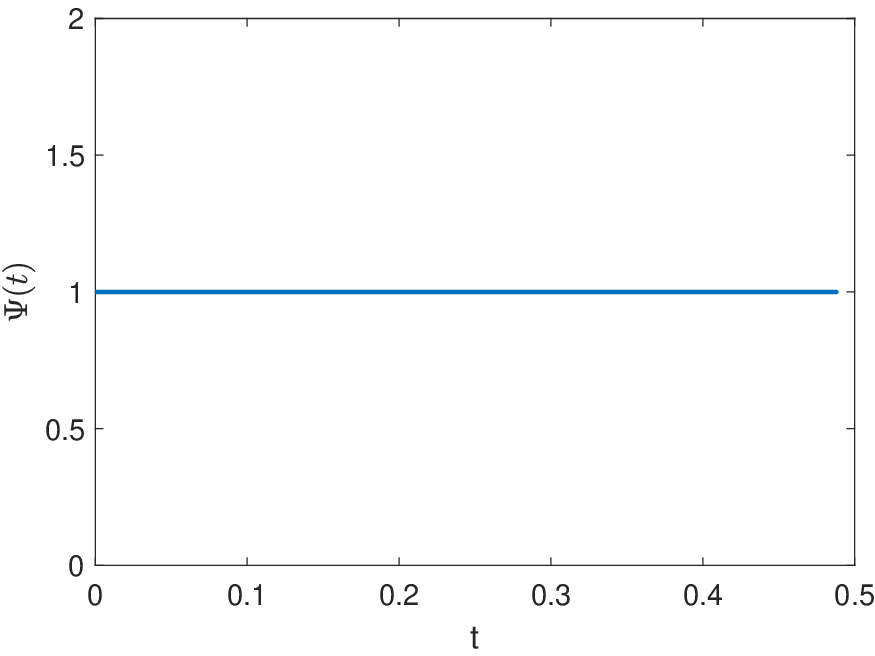}
	}
	\subfigure[Adding the penalty term to $\mathcal{E}_h(t)$.]{
		\includegraphics[width=2.8in]{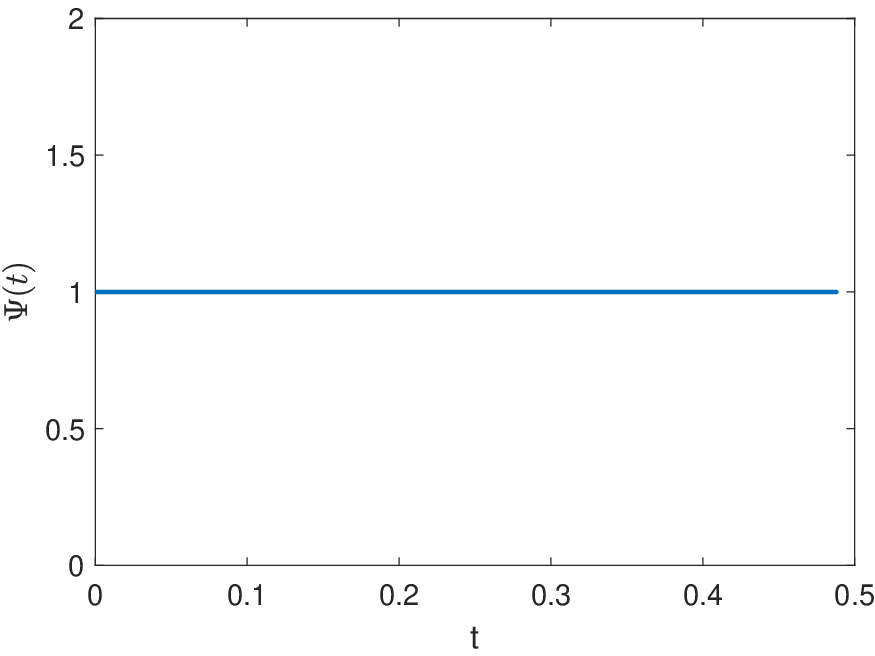}
	}
	\caption{The change of the MRI function $\Psi(t)$ over time for the circular curve.}
	\label{fig:penalty_circle}
\end{figure}

\begin{figure}[h]
	\centering
	\subfigure[$\mathcal{E}_h(t)$ does not include the penalty term.]{
		\includegraphics[width=2.8in]{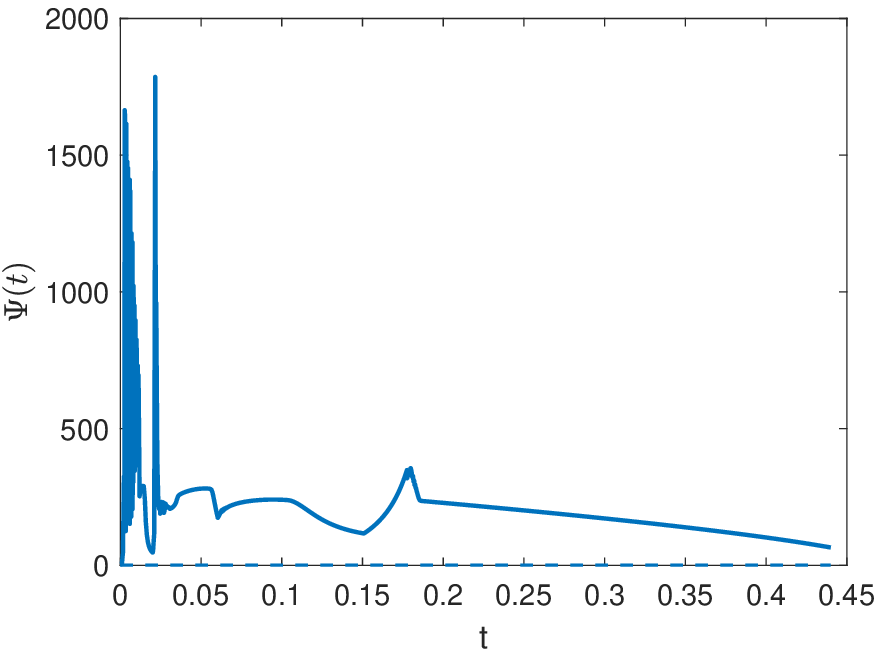}
	}
	\subfigure[Adding the penalty term to $\mathcal{E}_h(t)$.]{
		\includegraphics[width=2.8in]{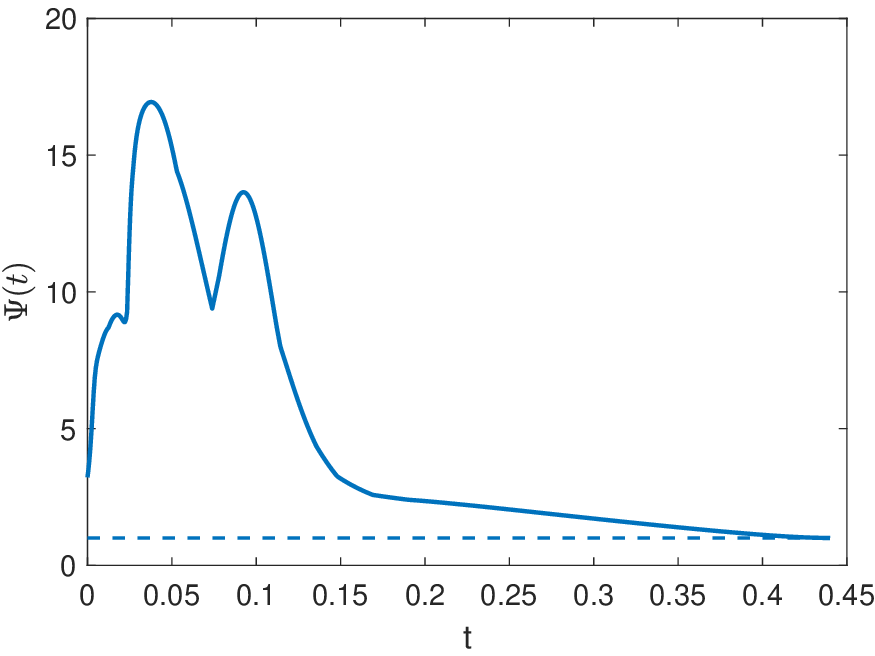}
	}
	\caption{The change of the MRI function $\Psi(t)$ over time for the flower-shaped curve.}
	\label{fig:penalty_flower}
\end{figure}

\subsection{Volume preserving mean curvature flow}
In this subsection, we show numerical simulations for the volume preserving mean curvature flow.
We choose the same initial curve as that in Section \ref{flower1}. 
We discretize the curve \eqref{flower equation} uniformly in $\alpha$ by setting $\x_i=\x(\alpha_i)$ with $\alpha_i=\frac{2\pi}{n}i$. The evolution of the curve is plotted in Fig. \ref{fig:flower2}. We can see that the flower-shaped curve slowly becomes convex  and eventually evolves into a circular shape.  
Due to the constraint that the enclosed area remains constant, the shape of the curve will not change after it becomes a circle, i.e., it will reach a stationary state. The discrete energy function $\mathcal{E}_h(t)$ changes with time as shown in Fig. \ref{fig: flower energy2}. We can see that the discrete energy function decreases in early stage and then approaches to a constant when the system goes to the stationary state.

\begin{figure}[h]
	\centering
	\subfigure[t=0]{
		\includegraphics[width=1.3in]{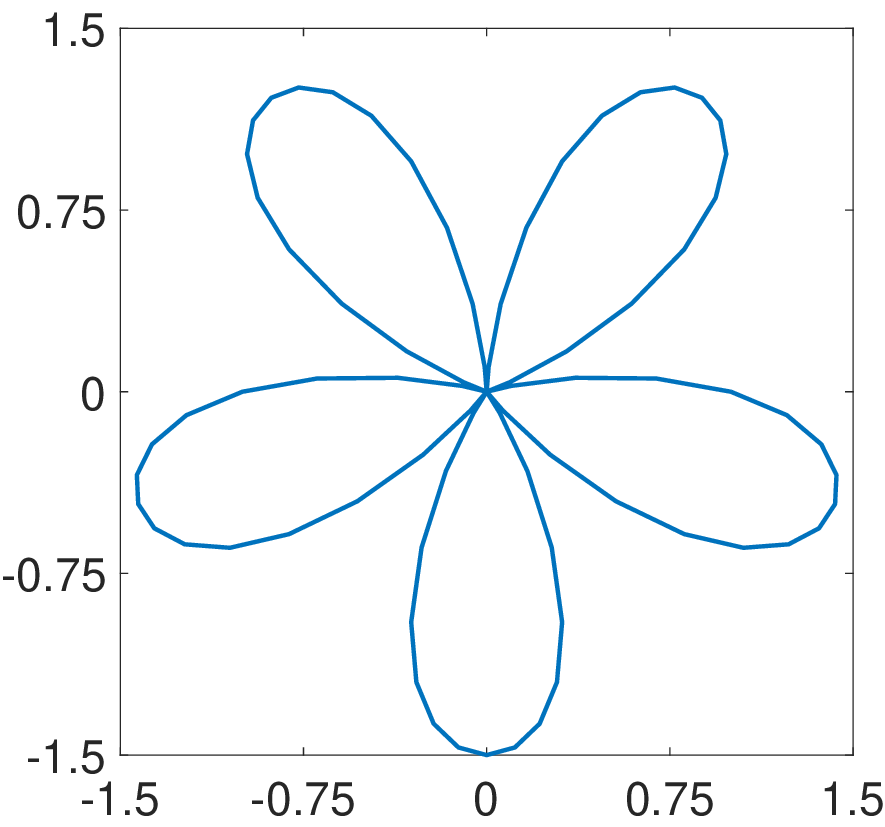}
	}
	\subfigure[t=0.03]{
		\includegraphics[width=1.3in]{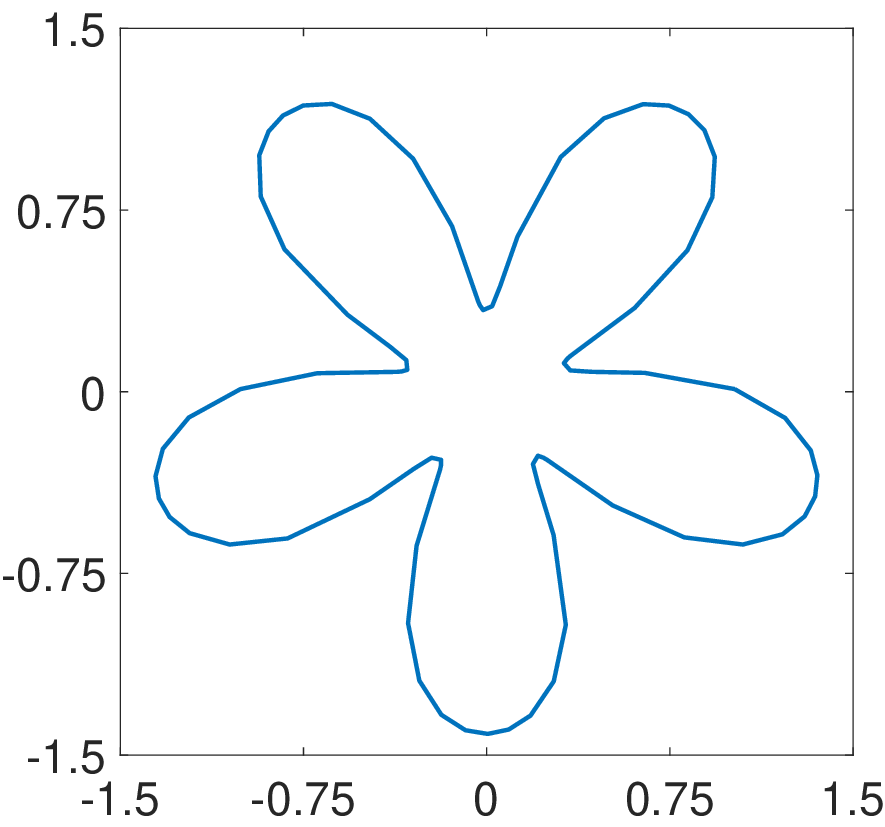}
	}
	\subfigure[t=0.06]{
		\includegraphics[width=1.3in]{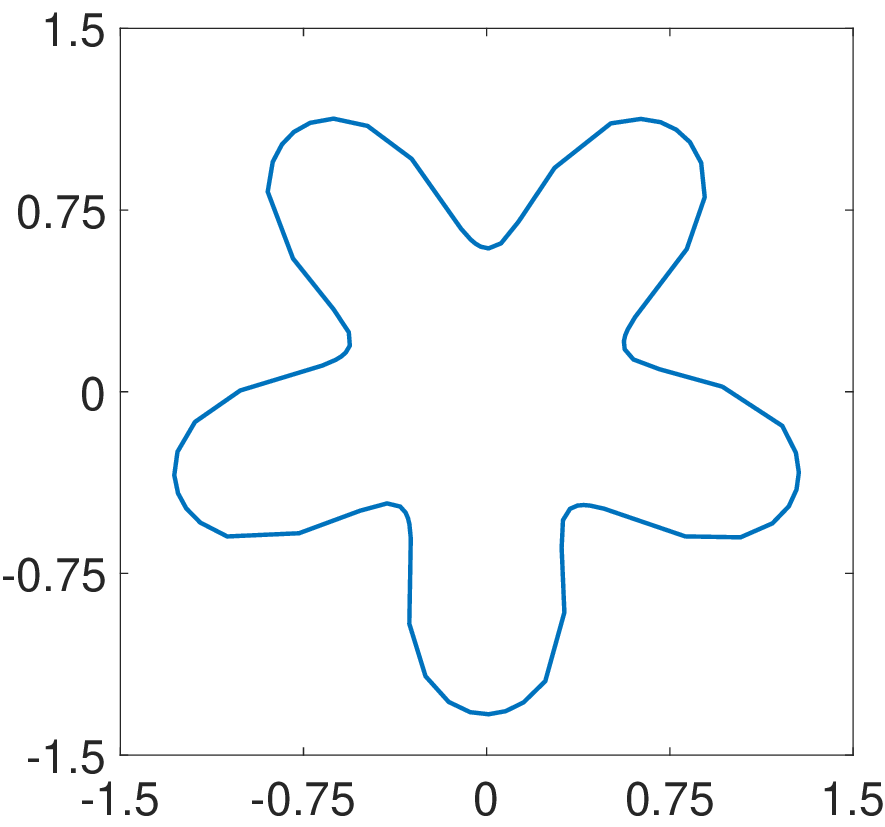}
	}
	
	\subfigure[t=0.09]{
		\includegraphics[width=1.3in]{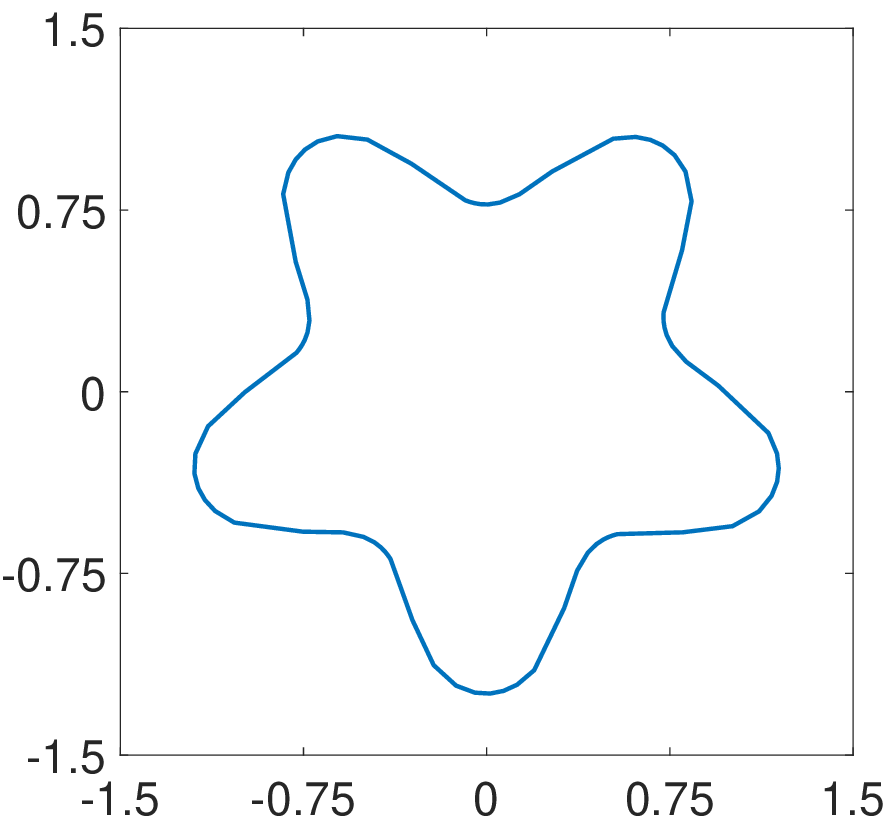}
	}
	\subfigure[t=0.12]{
		\includegraphics[width=1.3in]{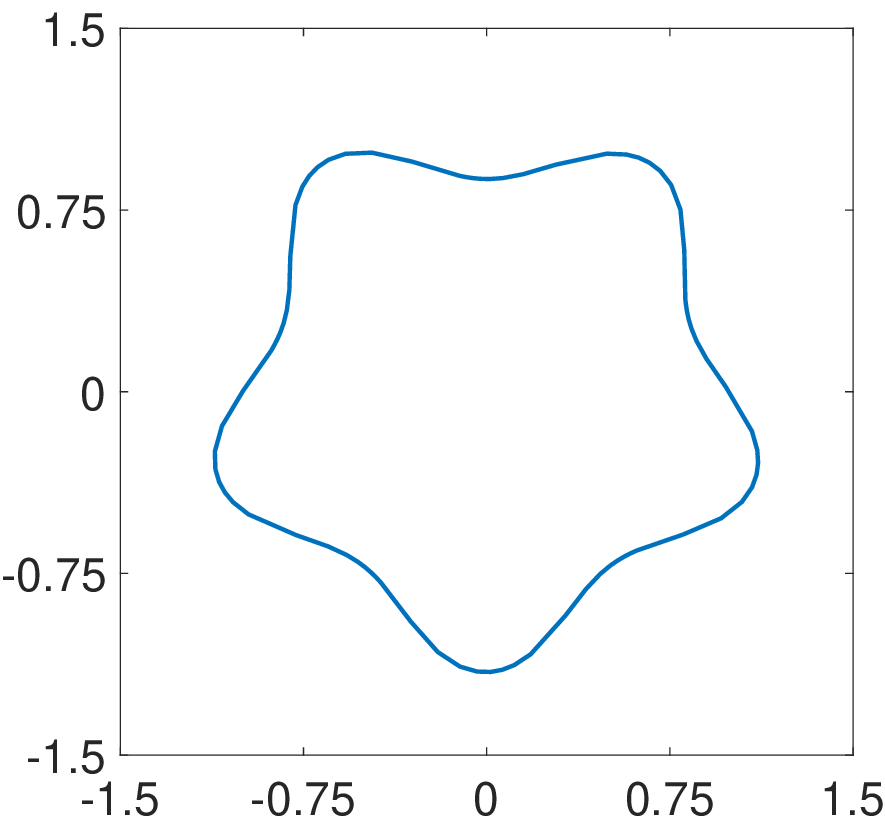}
	}
	\subfigure[t=0.15]{
		\includegraphics[width=1.3in]{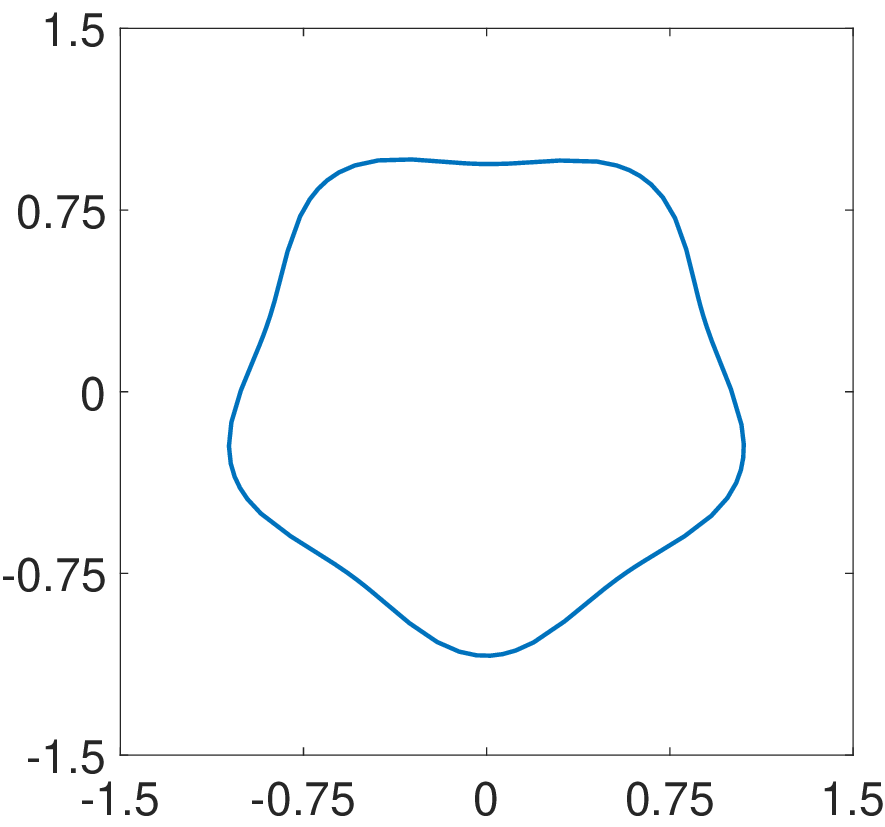}
	}
	
	\subfigure[t=0.18]{
		\includegraphics[width=1.3in]{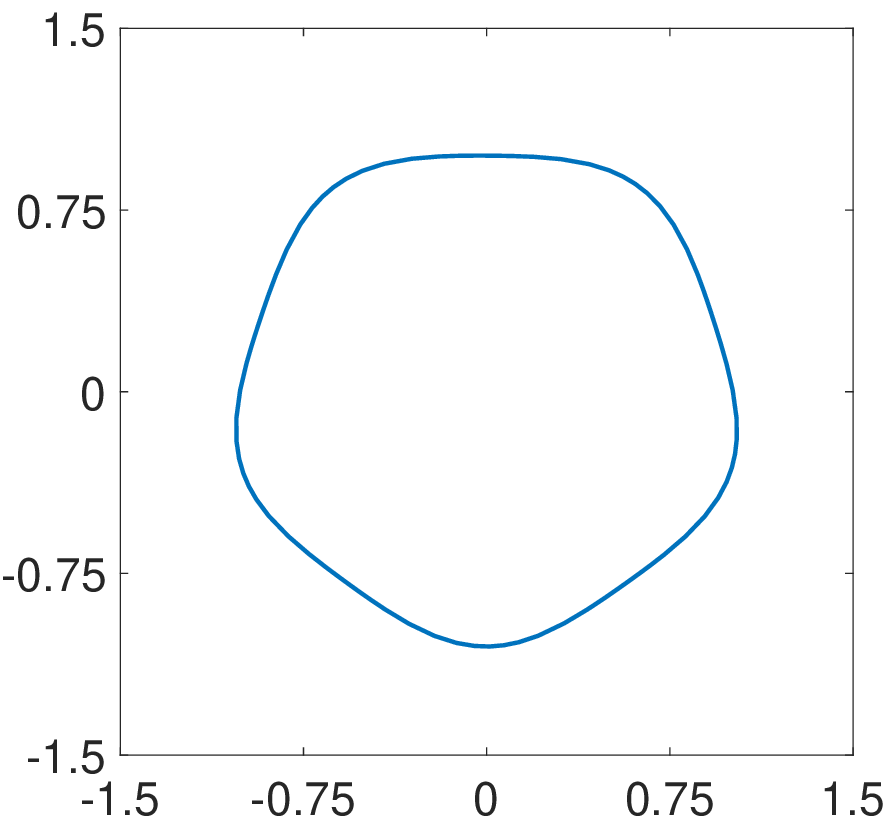}
	}
	\subfigure[t=0.3]{
		\includegraphics[width=1.3in]{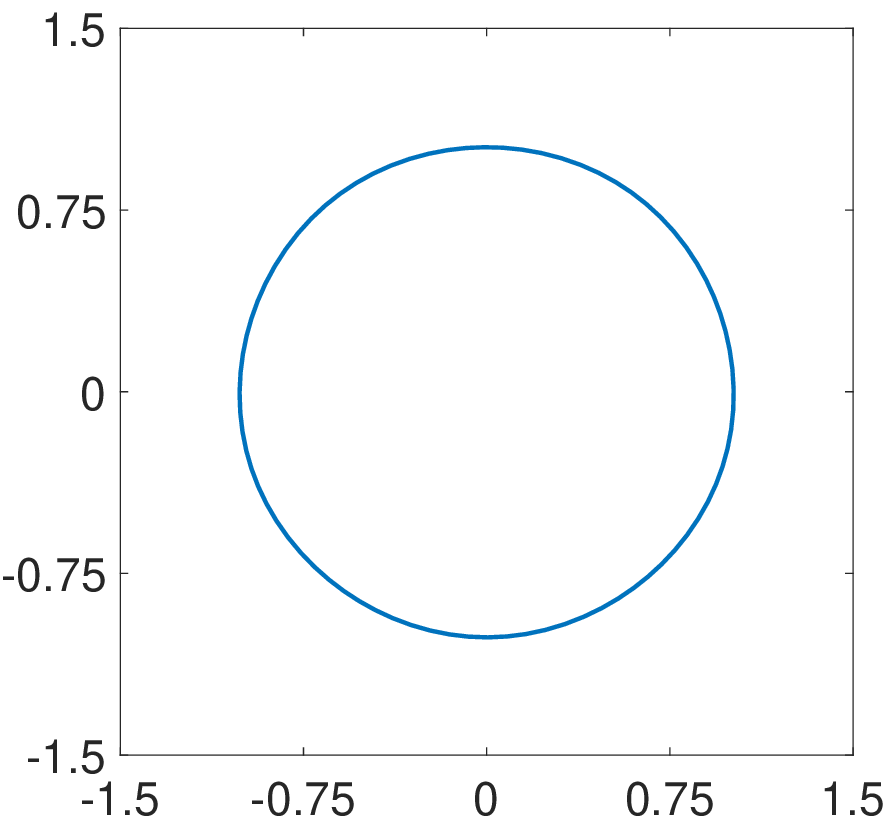}
	}
	\subfigure[t=0.4]{
		\includegraphics[width=1.3in]{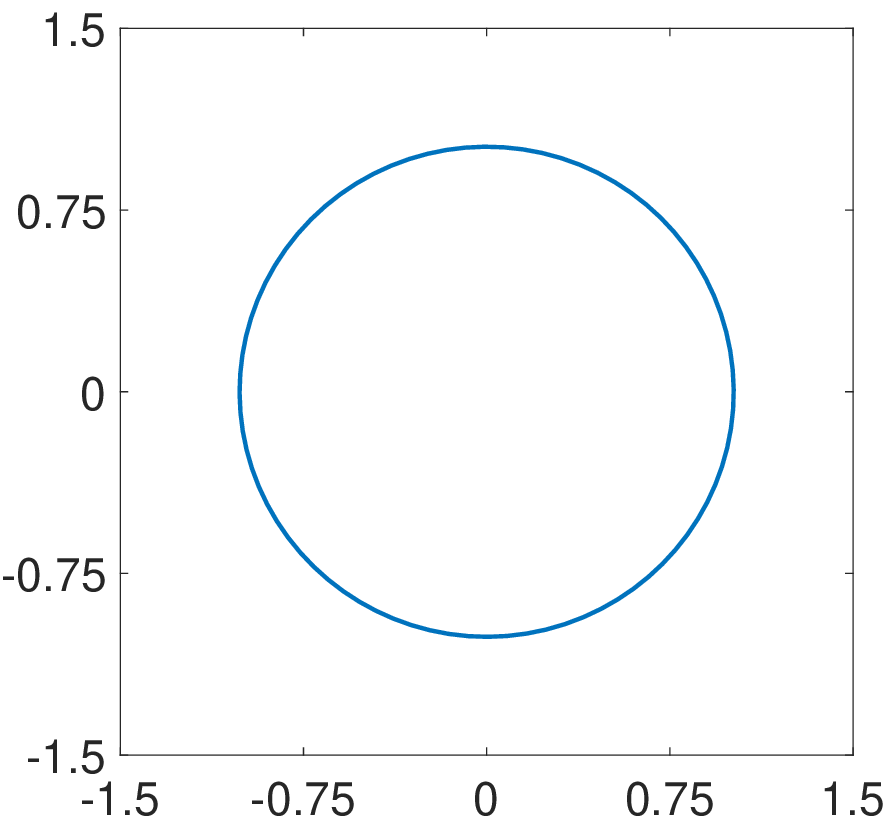}
	}
	\caption{Evolution of a flower-shaped curve under the volume preserving mean curvature flow (n=80).}
	\label{fig:flower2}
\end{figure}
\begin{figure}[h]
	\centering
	\includegraphics[width=0.5\textwidth]{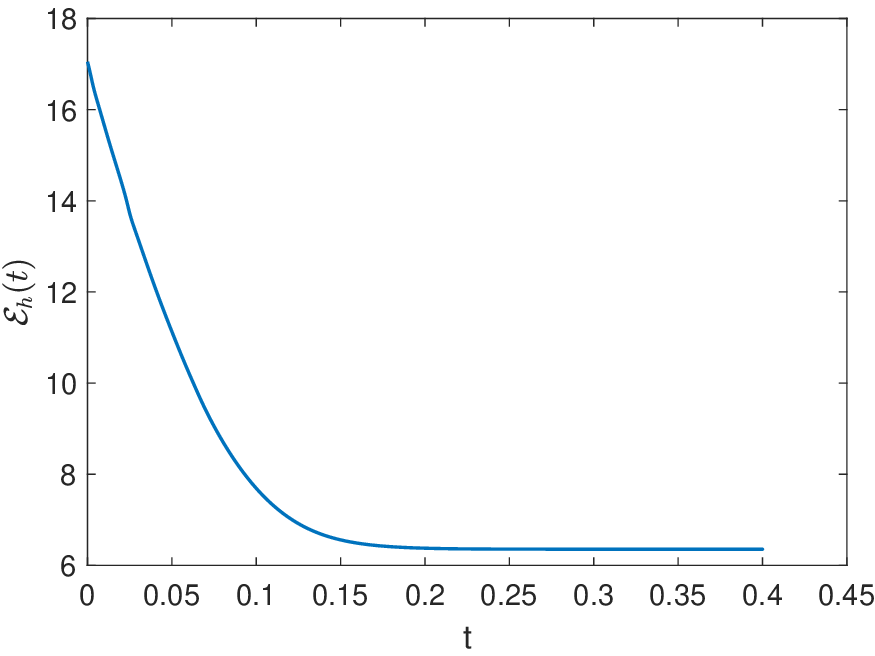}
	\caption{The change of the discrete energy $\mathcal{E}_h(t)$ with respect to time.}
	\label{fig: flower energy2}
\end{figure}

\subsection{Wetting problems}
We consider the $x$-axis as the horizontal line with the positive direction to the right. The initial curve is taken as a unit semicircle
\begin{equation} \label{semicircle}
	\tilde{\Gamma}(0): \x(\alpha)=(\cos\alpha,\sin\alpha)^\top, \quad 0\leq\alpha\leq\pi.
\end{equation}
Under the constraint that the area enclosed by $\tilde{\Gamma}(t)$ and the $x$-axis remains constant, when $\tilde{\Gamma}(t)$ reaches a stationary state under mean curvature flow, the energy no longer changes. From the expression \eqref{3E-L}, the left and right contact angles of the curve at stationary state are equal to the Young's angle.

We discretize the curve uniformly in $\alpha$ by setting $\x_i=\x(\alpha_i)$ with $\alpha_i=\frac{\pi}{n}i$ and take the Young's angle $\theta_Y$ as $\frac{2\pi}{3}$ and $\frac{\pi}{3}$ respectively. The evolution of the curve is plotted in Fig. \ref{fig:circle3} while the change of the discrete energy function $\tilde{\mathcal{E}}_h$ (with $C=0$ in Eq. \eqref{discrete_energy3}) over time is shown in Fig. \ref{fig:circle energy3}.  We can see that  the droplet gradually change its shape from a semi-circle to a circular shape with stationary contact angles.
From Fig. \ref{fig:circle energy3} we can see that the discrete energy function initially decreases to a constant value.

\begin{figure}[h]
	\centering
	\subfigure[$\theta_Y=\frac{2\pi}{3}$]{
		\includegraphics[width=3.5in]{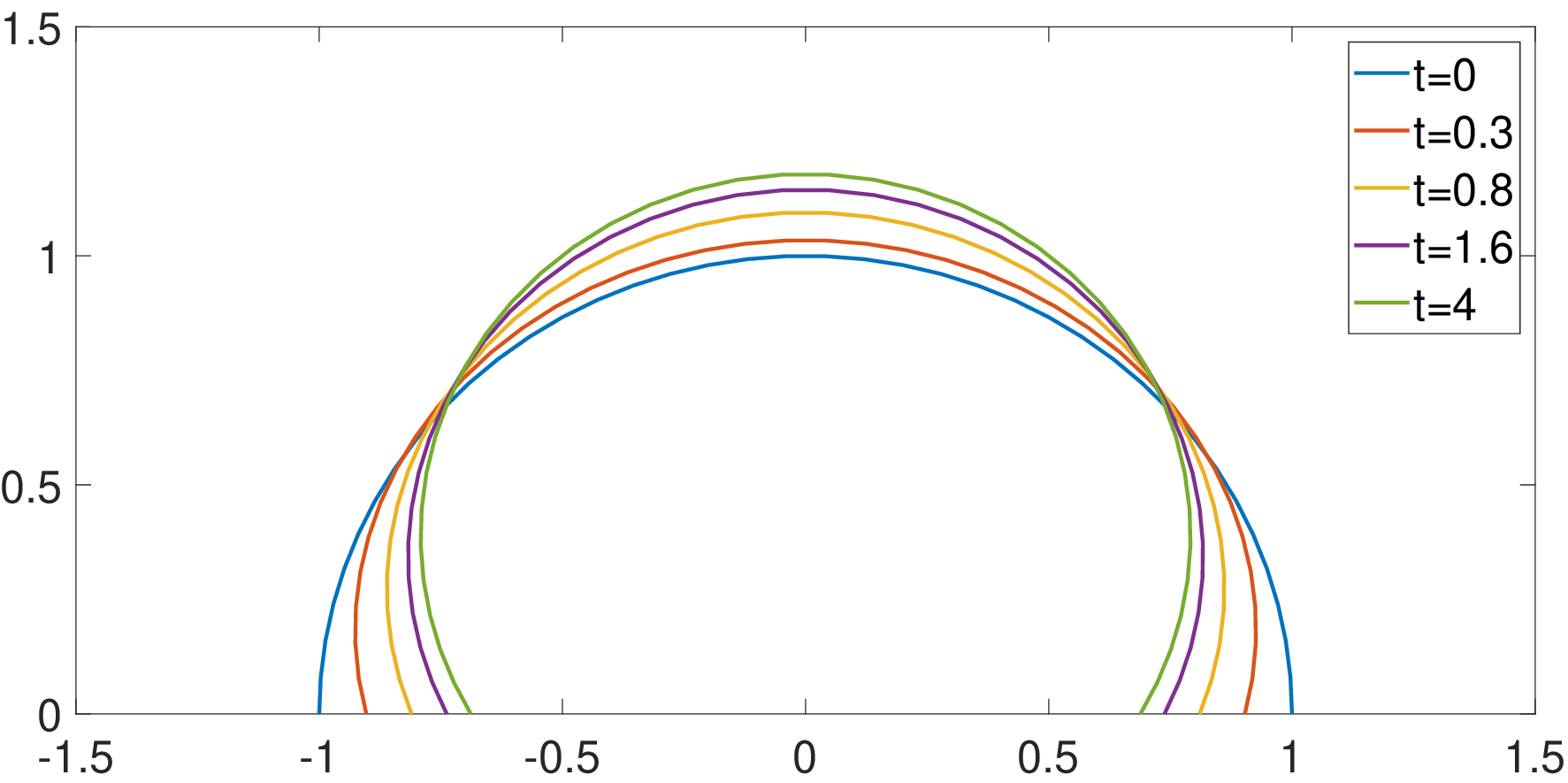}
	}
	
	\subfigure[$\theta_Y=\frac{\pi}{3}$]{
		\includegraphics[width=3.5in]{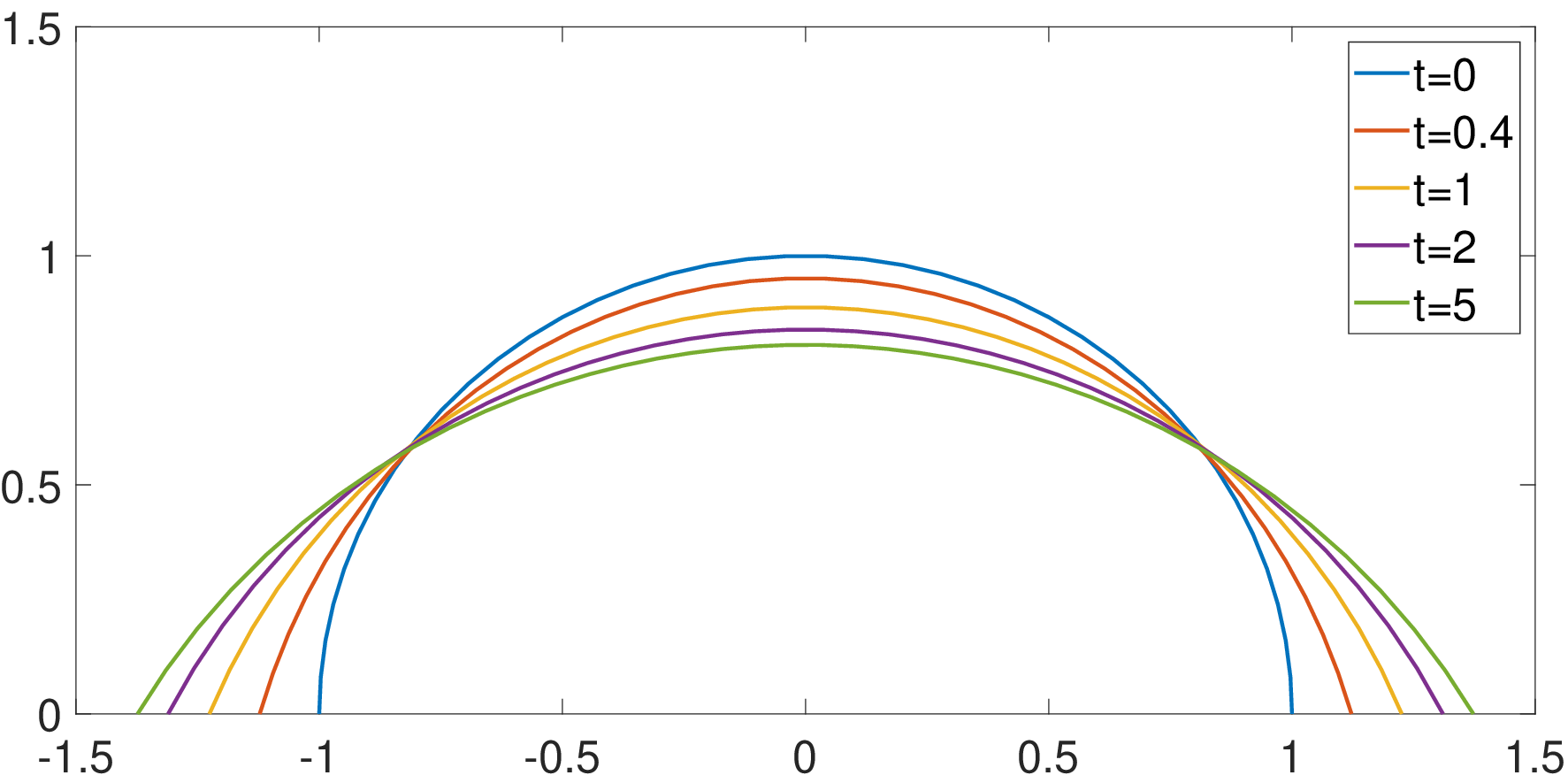}
	}
	\caption{Evolution of unit semicircles in the wetting problem (n=40).}
	\label{fig:circle3}
\end{figure}

\begin{figure}[h]
	\centering
	\subfigure[$\theta_Y=\frac{2\pi}{3}$]{
		\includegraphics[width=2.8in]{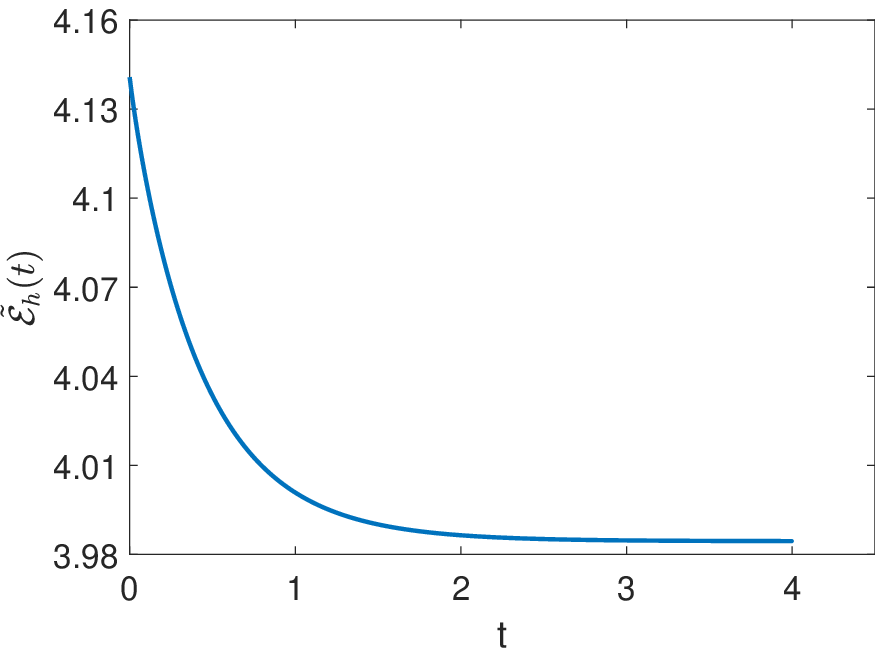}
	}
	\subfigure[$\theta_Y=\frac{\pi}{3}$]{
		\includegraphics[width=2.8in]{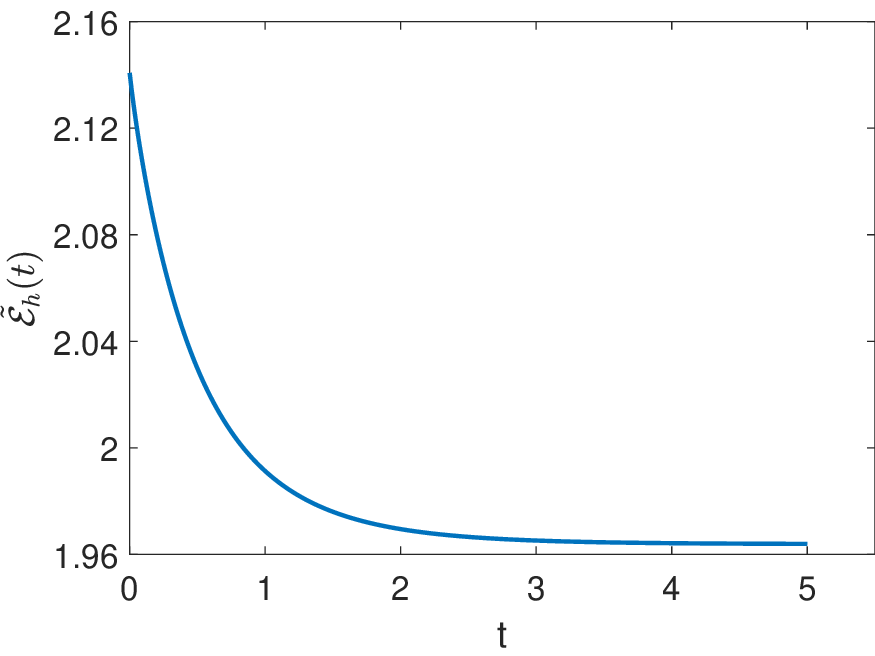}
	}
	\caption{The change of the discrete energy $\tilde{\mathcal{E}}_h(t)$ with respect to time.}
	\label{fig:circle energy3}
\end{figure}


The exact solutions of wetting problems can be affected by lateral displacement of the curve. We thus define the error as the difference of the stationary-state energy between the discrete curve and the smooth curve, i.e.
\begin{equation} \label{err_energy}
	err_{2}:=|\tilde{\mathcal{E}}-\tilde{\mathcal{E}}_h|.
\end{equation}
The convergence order is calculated as in Eq. \eqref{order}. We conducted numerical simulations with different parameters to evaluate the accuracy and convergence of the method, setting $\xi_0=\xi_1=1$. For the case where Young's angle is $\theta_Y=\frac{2\pi}{3}$, we use a time step size of $\Delta t=10^{-5}$ and varied the number of nodes as $n=5,10,20,40,80$. The simulation was performed until $T=4$. The resulting errors and convergence order are summarized in Table \ref{err order 120}. Similarly, for the case where Young's angle is $\theta_Y=\frac{\pi}{3}$, we used the same time step size and varied the number of nodes as before. The simulation was performed until $T=5$, and the corresponding errors and convergence order are presented in Table \ref{err order 60}. We could see the method has second order convergent with respect to the spacial mesh size. 
\begin{table}[htbp]   
	\begin{center}   
		\caption{ {The errors on the energy and the convergence order for  $\theta_Y=\frac{2\pi}{3}$}. }  
		\label{err order 120} 
		\begin{tabular}{c|c c c c c}   
			\hline \textbf{\em n} & 5 & 10 & 20 & 40 & 80 \\   
			\hline \textbf{\em $err_{2}$} & 5.5339e-02 & 1.0682e-02 & 2.3179e-03 & 5.0081e-04 & 9.2443e-05  \\ 
			\textbf{\em $Order$} & - & 2.37 & 2.20 & 2.21 & 2.44 \\      
			\hline    
		\end{tabular}   
	\end{center}   
\end{table}

\begin{table}[htbp]   
	\begin{center}   
		\caption{ {The errors on the energy and the convergence order for  $\theta_Y=\frac{\pi}{3}$}.}  
		\label{err order 60} 
		\begin{tabular}{c|c c c c c}   
			\hline \textbf{\em n} & 5 & 10 & 20 & 40 & 80 \\   
			\hline \textbf{\em $err_2$} & 6.2309e-02 & 1.2166e-02 & 2.4918e-03 & 4.8600e-04 & 5.6509e-05  \\ 
			\textbf{\em $Order$} & - & 2.36 & 2.29 & 2.36 & 3.10  \\      
			\hline    
		\end{tabular}   
	\end{center}   
\end{table}

Notice that the energy is a very weak measure on the accuracy of the numerical solution. We could also consider the convergence of the error under some stronger measures, like the manifold distance \cite{zhao2021energy}, i.e.
\begin{equation} \label{err_area}
	err_3 := |(\Omega_h \backslash \Omega)\cup(\Omega \backslash \Omega_h)| = |\Omega_h| + |\Omega| - 2|(\Omega_h \cap \Omega)|. 
\end{equation}
Here $\Omega_h$ is the region enclosed by the discrete curve $\tilde{\Gamma}_h$ and the horizontal line, and $\Omega$ is the region enclosed by the smooth curve $\tilde{\Gamma}$ and the horizontal line. However, we find that the convergence rate is not optimal for the error. The main reason is that we are comparing with the analytical solution in stationary state, which is obtained by the numerical solution at final time $T$ even the energy seems unchanged anymore. To accelerate the process of the droplet going to stationary state, we decrease the friction coefficients  $\xi_0$ and $\xi_1$. For the case $\theta_Y = \frac{2\pi}{3}$, we set $\xi_0 = \xi_1 = 0.1$, with a time step size of $\Delta t=5\times10^{-6}$. For the case $\theta_Y = \frac{\pi}{3}$, we set $\xi_0 = \xi_1 = 0.02$, with a time step size of $\Delta t=10^{-6}$. The end time is the same as the previous tests. The numerical errors and convergence order are given in Tables \ref{err_area order 120} and \ref{err_area order 60}.  We can see that our method  gives optimal convergence order with respect to spatial mesh size for the errors measured in manifold distance.

\begin{table}[htbp]   
	\begin{center}   
		\caption{The errors in manifold distance and the convergence order for  $\theta_Y=\frac{2\pi}{3}$.}  
		\label{err_area order 120} 
		\begin{tabular}{c|c c c c c}   
			\hline \textbf{\em n} & 5 & 10 & 20 & 40 & 80 \\   
			\hline \textbf{\em $err_3$} & 1.7023e-01 & 3.4356e-02 & 7.6659e-03 & 1.7409e-03 & 4.4573e-04  \\ 
			\textbf{\em $Order$} & - & 2.31 & 2.16 & 2.14 & 1.97 \\      
			\hline    
		\end{tabular}   
	\end{center}   
\end{table}

\begin{table}[htbp]   
	\begin{center}   
		\caption{The errors in manifold distance and the convergence order for  $\theta_Y=\frac{\pi}{3}$.}  
		\label{err_area order 60} 
		\begin{tabular}{c|c c c c c}   
			\hline \textbf{\em n} & 5 & 10 & 20 & 40 & 80 \\   
			\hline \textbf{\em $err_3$} & 1.5658e-01 & 3.1706e-02 & 7.1495e-03 & 1.6998e-03 & 4.1545e-04  \\ 
			\textbf{\em $Order$} & - & 2.30 & 2.15 & 2.07 & 2.03  \\      
			\hline    
		\end{tabular}   
	\end{center}   
\end{table}

\section{Conclusions}
In this paper, we develop a novel finite element method for mean curvature flows by using the Onsager principle as an approximation tool. The key feature of the method is that the semi-discrete scheme preserves the energy dissipation relations of the Onsager principle. This is important for many applications. We apply the method to three different problems, namely the simple mean curvature flow, the volume preserving mean curvature flow and a wetting problem on substrate. Numerical examples show that
the method works well for all the three problems. 

There are several problems we need to study in the future. Firstly, we apply an explicit Euler scheme for the time discretization. The scheme is stable only the time step is small enough. We need consider implicit or semi-implicit scheme to improve the stability of the scheme. Another important issue is on the quality of meshes of the discrete geometric flow. To avoid mesh degeneracy and entanglement, we add a penalty term to the energy. The technique works well only when we need choose the penalty parameter and the time step properly. It is known that some other methods have properties to keep the mesh equally distributed in arc length. It is interesting to combine these techniques with our method. In addition, it is interesting to do numerical analysis of the method and also extend the method to higher order approximations and other geometric flows.

\section*{Acknowledgments}
The work was partially supported by NSFC 11971469 and 12371415. We would also like to acknowledge the support of the Beijing Natural Science Foundation (Grant No. Z240001). Additionally, we thank Professor Wei Jiang for his helpful discussions.

\appendix
\section*{Appendix}
\setcounter{equation}{0}
\renewcommand{\theequation}{A.\arabic{equation}}

In sections \ref{problem 1}-\ref{problem 3}, we derived multiple systems of equations. For convenience of the reader, we will show the specific form of Eqs. \eqref{AV=G}, \eqref{AV=G penalty term}, \eqref{2AV=G} and \eqref{3AV=G} below.

Firstly, for Eq. \eqref{AV=G}, the coefficient matrix $A$ and the right-hand term $G$ are computed as follows: 
$$
\begin{gathered}
	A=
	\begin{pmatrix}
		A_{11}&A_{12}& & &A_{1n}\\
		A_{21}&A_{22}&A_{23}\\
		\quad&\ddots&\ddots&\ddots\\
		& &A_{n-1,n-2}&A_{n-1,n-1}&A_{n-1,n}\\
		A_{n,1}& & &A_{n,n-1}&A_{n,n}
	\end{pmatrix},
\end{gathered}
$$
where the elements $A_{ij}\ (1\leq i,j\leq n)$ are square matrixes of order 2. Denote by $I_2$ the unit matrix of order 2. Then we have
\begin{equation} \label{a_ii}
	\begin{aligned}
		& A_{ii}=\frac{|\x_{i}(t)-\x_{i-1}(t)| + |\x_{i+1}(t)-\x_{i}(t)|}{3} I_2, &1\leq i\leq n-1, \\
		& A_{i,i+1}=A_{i+1,i}=\frac{| \x_{i+1}(t)-\x_{i}(t) |}{6} I_2, &1\leq i\leq n-1,  
	\end{aligned}
\end{equation}
and
\begin{equation} \label{a_nn}
	\begin{aligned}
		& A_{nn}=\frac{| \x_{n}(t)-\x_{n-1}(t) |+| \x_{1}(t)-\x_{n}(t) |}{3} I_2,\\
		& A_{n,1}=A_{1,n}=\frac{| \x_{1}(t)-\x_{n}(t) |}{6} I_2.
	\end{aligned}
\end{equation}
The right-hand term $G=(\vec{g}_1^\top,\vec{g}_2^\top,...,\vec{g}_n^\top)^\top$ with 
\begin{equation} \label{g_i}
	\vec{g}_i=-\frac{\partial \mathcal{E}_h}{\partial \x_i}=\frac{\x_{i+1}(t)-\x_{i}(t)}{| \x_{i+1}(t)-\x_{i}(t) |}-\frac{\x_{i}(t)-\x_{i-1}(t)}{| \x_{i}(t)-\x_{i-1}(t) |}, \quad 1\leq i\leq n-1, 
\end{equation}
and
\begin{equation} \label{g_n}
	\vec{g}_n=-\frac{\partial \mathcal{E}_h}{\partial \x_n}=\frac{\x_{1}(t)-\x_{n}(t)}{| \x_{1}(t)-\x_{n}(t) |}-\frac{\x_{n}(t)-\x_{n-1}(t)}{| \x_{n}(t)-\x_{n-1}(t) |}.
\end{equation}
By applying mass lumping to Eq. \eqref{AV=G}, we obtain the classical semi-discrete Dziuk scheme.

Secondly, for Eq. \eqref{AV=G penalty term}, the right-hand term $G^{\delta}=((\vec{g}_1^{\delta})^\top,(\vec{g}_2^{\delta})^\top,...,(\vec{g}_n^{\delta})^\top)^\top$ is computed as follows: 
\begin{equation*}
	\begin{aligned}
		g_i^{\delta}=-\frac{\partial \mathcal{E}_h^{\delta}}{\partial \x_i}
		&=\frac{\x_{i+1}-\x_{i}}{| \x_{i+1}-\x_{i} |}-\frac{\x_{i}-\x_{i-1}}{| \x_{i}-\x_{i-1} |}\\
		&\ +2\delta(\frac{| \x_{i+1}-\x_{i} |}{| \x_{i+2}-\x_{i+1} |}-1)\frac{\x_{i+1}-\x_{i}}{| \x_{i+2}-\x_{i+1} | \cdot | \x_{i+1}-\x_{i} |}\\
		&\ -2\delta(\frac{| \x_{i}-\x_{i-1} |}{| \x_{i+1}-\x_{i} |}-1)\frac{\frac{\x_{i}-\x_{i-1}}{| \x_{i}-\x_{i-1} |}| \x_{i+1}-\x_{i} | + | \x_{i}-\x_{i-1} | \frac{\x_{i+1}-\x_{i}}{| \x_{i+1}-\x_{i} |}}{| \x_{i+1}-\x_{i} |^2}\\
		&\ +2\delta(\frac{| \x_{i-1}-\x_{i-2} |}{| \x_{i}-\x_{i-1} |}-1)\frac{| \x_{i-1}-\x_{i-2} |(\x_{i}-\x_{i-1})}{| \x_{i}-\x_{i-1} |^3}.
	\end{aligned}
\end{equation*}

Next, for Eq. \eqref{2AV=G}, the coefficient matrix $\hat{A}$ and the right-hand term $\hat{G}$ are computed as follows: 
$$
\begin{gathered}
	\hat{A}=
	\begin{pmatrix}
		A_{11}&A_{12}& & & & & &A_{1n}&\vec{b}_{1,n+1}\\
		A_{21}&A_{22}&A_{23}\\
		\quad&\ddots&\ddots&\ddots& & & & &\vdots\\
		& &A_{i,i-1}&A_{i,i}&A_{i,i+1}\\
		& & &A_{i+1,i}&A_{i+1,i+1}&A_{i+1,i+2}\\
		& & & &\ddots&\ddots&\ddots& &\vdots\\
		& & & & &A_{n-1,n-2}&A_{n-1,n-1}&A_{n-1,n}&\vec{b}_{n-1,n+1}\\
		A_{n,1}& & & & & &A_{n,n-1}&A_{n,n}&\vec{b}_{n,n+1}\\
		\vec{b}_{1,n+1}^\top& &\ldots& & &\ldots&\vec{b}_{n-1,n+1}^\top&\vec{b}_{n,n+1}^\top&0
	\end{pmatrix},
\end{gathered}
$$
and $\hat{G}=(\vec{g}_{1}^\top,\vec{g}_{2}^\top,...,\vec{g}_{n}^\top,g_{n+1})^\top$. Where the formulations of $A_{ij}\ (1\leq i,j \leq n)$ and $\vec{g}_i\ (1\leq i \leq n)$ are given in Eqs. \eqref{a_ii}, \eqref{a_nn}, and \eqref{g_i}, \eqref{g_n}, respectively. In addition, we have $g_{n+1}=0$. The elements in last column of $\hat{A}$ are given by
\begin{equation*}
	\begin{aligned}	
		& \vec{b}_{i,n+1} = \frac{P}{2}[\x_{i+1}(t)-\x_{i-1}(t)], \quad for\quad 2\leq i\leq n-1, \\
		& \vec{b}_{1,n+1} = \frac{P}{2}[\x_2(t)-\x_n(t)],\quad \vec{b}_{n,n+1} = \frac{P}{2}[\x_1(t)-\x_{n-1}(t)]. 
	\end{aligned}
\end{equation*}

Finally, for Eq. \eqref{3AV=G}, the coefficient matrix $\tilde{A}$ and the right-hand term $\tilde{G}$ are computed as follows: 
$$
\begin{gathered}
	\tilde{A}=
	\begin{pmatrix}
		\tilde{A}_{11}&\tilde{A}_{12}& & & & & & &\tilde{\vec{b}}_{1,n+1}\\
		\tilde{A}_{21}&\tilde{A}_{22}&\tilde{A}_{23}\\
		\quad&\ddots&\ddots&\ddots& & & & &\vdots\\
		& &\tilde{A}_{i,i-1}&\tilde{A}_{i,i}&\tilde{A}_{i,i+1}\\
		& & &\tilde{A}_{i+1,i}&\tilde{A}_{i+1,i+1}&\tilde{A}_{i+1,i+2}\\
		& & & &\ddots&\ddots&\ddots& &\vdots\\
		& & & & &\tilde{A}_{n-1,n-2}&\tilde{A}_{n-1,n-1}&\tilde{A}_{n-1,n}&\tilde{\vec{b}}_{n-1,n+1}\\
		& & & & & &\tilde{A}_{n,n-1}&\tilde{A}_{n,n}&\tilde{\vec{b}}_{n,n+1}\\
		\tilde{\vec{b}}^\top_{1,n+1}& &\ldots& & &\ldots&\tilde{\vec{b}}^\top_{n-1,n+1}&\tilde{\vec{b}}^\top_{n,n+1}&0
	\end{pmatrix},
\end{gathered}
$$
where $\tilde{A}_{ij}\ (1\leq i,j\leq n)$ are square matrices of order 2, and $\tilde{\vec{b}}_{i,n+1}\ (1\leq i\leq n)$ are two dimensional column vectors. The subblock on the diagonal of $\tilde{A}$ are computed as
\begin{equation*}
	\begin{aligned}
		&\tilde{A}_{ii}=\xi_0\frac{| \x_{i}(t)-\x_{i-1}(t) |+| \x_{i+1}(t)-\x_{i}(t) |}{3}I_2, \quad  2\leq i\leq n-1, \\ 
		&\tilde{A}_{11}=\left( \begin{array}{cc} \xi_0\frac{|\x_2(t)-\x_1(t)|}{3}+\xi_1&0\\0&1 \end{array} \right),\quad
		\tilde{A}_{nn}=\left( \begin{array}{cc} \xi_0\frac{|\x_n(t)-\x_{n-1}(t)|}{3}+\xi_1&0\\0&1 \end{array} \right).
	\end{aligned}
\end{equation*}
The subblocks on the superdiagonal and subdiagonal of $\tilde{A}$ are computed as 
\begin{equation*}
	\begin{aligned}
		&\tilde{A}_{i,i+1}=\tilde{A}_{i+1,i}^\top=\xi_0\frac{| \x_{i+1}(t)-\x_{i}(t) |}{6}I_2, \quad 2\leq i\leq n-2, \\ 
		&\tilde{A}_{21}=\tilde{A}_{12}^\top=\left( \begin{array}{cc} \xi_0\frac{|\x_2(t)-\x_1(t)|}{6}&0\\0&0 \end{array} \right),\quad
		\tilde{A}_{n-1,n}=\tilde{A}_{n,n-1}^\top=\left( \begin{array}{cc} \xi_0\frac{|\x_n(t)-\x_{n-1}(t)|}{6}&0\\0&0 \end{array} \right).
	\end{aligned}
\end{equation*}
The last column of $\tilde{A}$ are calculated as
\begin{equation*}
	\begin{aligned}
		&\tilde{\vec{b}}_{i,n+1} = \frac{P}{2}[\x_{i+1}(t)-\x_{i-1}(t)], \quad 2\leq i\leq n-1, \\
		&\tilde{\vec{b}}_{1,n+1} = (\frac{x_2^{(2)}(t)}{2}, 0)^\top,\quad
		\tilde{\vec{b}}_{n,n+1} = (-\frac{x_{n-1}^{(2)}(t)}{2}, 0)^\top.
	\end{aligned}
\end{equation*}
The right-hand term $\tilde{G}=(\tilde{\vec{g}}_1^\top,\tilde{\vec{g}}_2^\top,...,\tilde{\vec{g}}_n^\top,\tilde{g}_{n+1})^\top$, where $\tilde{\vec{g}}_i\ (1\leq i\leq n)$ denote two dimensional column vectors and $\tilde{g}_{n+1}$ denotes a scalar, respectively. They are calculated as 
\begin{equation*}
	\begin{aligned}
		&\tilde{\vec{g}}_i=-\frac{\partial \tilde{\mathcal{E}}_h}{\partial \x_i}=\gamma[\frac{\x_{i+1}(t)-\x_{i}(t)}{| \x_{i+1}(t)-\x_{i}(t) |}-\frac{\x_{i}(t)-\x_{i-1}(t)}{| \x_{i}(t)-\x_{i-1}(t) |}], &2\leq i\leq n-1, \\
		&\tilde{\vec{g}}_1=(-\frac{\partial \tilde{\mathcal{E}}_h}{\partial x_1^{(1)}},0)^\top=(\gamma[\frac{x_2^{(1)}(t)-x_1^{(1)}(t)}{| \x_{2}(t)-\x_{1}(t) |}+\cos\theta_Y], 0)^\top,\\
		&\tilde{\vec{g}}_n=(-\frac{\partial \tilde{\mathcal{E}}_h}{\partial x_n^{(1)}},0)^\top=(-\gamma[\frac{x_n^{(1)}(t)-x_{n-1}^{(1)}(t)}{| \x_{n}(t)-\x_{n-1}(t) |}+\cos\theta_Y], 0)^\top,\\
		&\tilde{g}_{n+1}=0, &i=n+1.
	\end{aligned}
\end{equation*}

	\bibliographystyle{abbrv}  
	\bibliography{reference.bib}

\end{document}